\documentclass{amsart}
\usepackage{hyperref}
\hypersetup{
	colorlinks=true,
    linkcolor=black,
    citecolor=black,
    filecolor=black,
    urlcolor=black,
}
\usepackage{verbatim}
\usepackage{xy}
\usepackage{amsfonts}
\input xy
\xyoption{all}
\usepackage[shortlabels]{enumitem}
\usepackage{tikz-cd}
\usepackage{Definitions}
\usepackage{Environments2}
\usepackage{PageSetup}
\usepackage{amssymb}
\usepackage{url}
\usepackage{geometry} 
\usepackage[toc,page]{appendix} 

\usepackage[ 
color=orange!80, 
bordercolor=black,
textwidth=.8in,
textsize=small]
{todonotes}

\makeatletter \providecommand\@dotsep{5}
\makeatother

\newcommand{\prank}{\ensuremath{\mathrm{rk}_p}}
\newcommand{\coprank}{\ensuremath{\mathrm{cork}_p}}
\renewcommand{\sp}{\mathrm{Sp}}
\newcommand{\Spc}{\mathrm{Spc}}

\begin{document}
	\title[The Balmer spectrum of genuine $A$-spectra]{The Balmer spectrum of the equivariant homotopy category of a finite abelian group\vspace{-1ex}}
	\author[Barthel]{Tobias Barthel}
	\address{University of Copenhagen\\ Copenhagen, Denmark} 
	\email{tbarthel@math.ku.dk}
	
	\author[Hausmann]{Markus Hausmann}
	\address{University of Copenhagen\\ Copenhagen, Denmark} 
	\email{hausmann@math.ku.dk}

	\author[Naumann]{Niko Naumann}
	\address{University of Regensburg\\
	NWF I - Mathematik; Regensburg, Germany}
	\email{Niko.Naumann@mathematik.uni-regensburg.de}
	\urladdr{http://homepages.uni-regensburg.de/~nan25776/}

	\author[Nikolaus]{Thomas Nikolaus}
	\address{Max Planck Institute for Mathematics}
	\email{thoni@mpim-bonn.mpg.de}

	\author[Noel]{Justin Noel}
	\address{University of Regensburg\\
	NWF I - Mathematik; Regensburg, Germany}
	\email{justin.noel@mathematik.uni-regensburg.de}
	\urladdr{http://nullplug.org}

	\author[Stapleton]{Nathaniel Stapleton\vspace{-3ex}}
	\address{University of Regensburg\\ NWF I - Mathematik; Regensburg, Germany} 
	\email{nat.j.stapleton@gmail.com}
	\thanks{Justin Noel was partially supported by the DFG grants: NO 1175/1-1 and SFB 1085 - Higher Invariants, Regensburg. Niko Naumann and Nathaniel Stapleton were also partially supported by the SFB 1085 - Higher Invariants, Regensburg. Tobias Barthel and Markus Hausmann were supported by the DNRF92.}


\begin{abstract}
	For a finite abelian group $A$, we determine the Balmer spectrum of $\sp_A^{\omega}$, the compact objects in genuine $A$-spectra. This generalizes the case $A=\mathbb{Z}/p\mathbb{Z}$ due to Balmer and Sanders \cite{Balmer-Sanders}, by establishing (a corrected version of) their log$_p$-conjecture for abelian groups. We also work out the consequences for the chromatic type of fixed-points and establish a generalization of Kuhn's blue-shift theorem for Tate-constructions \cite{kuhn}.\vspace{-8ex}
\end{abstract}

	\maketitle

\tableofcontents

\section{Introduction}
\subsection*{Extended Abstract}

In \cite{balmer-the-spectrum}, Balmer constructs a topological space $\Spc(\cT)$, called the Balmer spectrum of $\cT$, for any essentially small triangulated category $\cT$ equipped with a compatible symmetric monoidal structure. This theory unifies the general reconstruction theorems for quasi-compact quasi-separated schemes (e.g.~\cite{thomason,balmer_reconstruction}) and the notion of support varieties appearing in modular representation theory (e.g.~\cite{bcr_modular, bik_modular}). Moreover, the primordial examples, namely the thick subcategory theorem of Hopkins and Smith~\cite{Hopkins-Smith} for spectra and the analogous result for the derived category of a Noetherian commutative ring by Hopkins and Neeman~\cite{Hop87,Nee92}, fit naturally into this framework. In each case, the complete description of $\Spc(\cT)$  was a major breakthrough in the respective field, as the space $\Spc(\cT)$ captures the global structure of $\cT$. 

This paper is concerned with the Balmer spectrum of (the homotopy category of) $\sp_G^\omega$ of compact genuine $G$-spectra \cite{LMS86} for a finite group $G$. This category blends topological information such as the stable homotopy groups of spheres with group-theoretic information such as Burnside rings and group cohomology. It has been much studied in recent years, especially after its crucial role in the solution of the Kervaire invariant one problem~\cite{hhr_kervaire}. Building on unpublished work of Strickland and Joachimi~\cite{stroilimi}, Balmer and Sanders \cite{Balmer-Sanders} determine the underlying \emph{set} of $\Spc(\sp_G^\omega)$  and show that the \emph{topology} of $\Spc(\sp_G^\omega)$ is closely related to the blue-shift phenomenon in generalized Tate cohomology discovered by Greenlees, Hovey, and Sadofsky~\cite{Greenlees-Sadofsky, Hovey-Sadofsky}. More precisely, they show that determining the topology on $\Spc(\sp_G^\omega)$ is equivalent to computing the blue-shift numbers $\beth_n(G;-,-)$ (defined below in \cref{def:blue-shift-numbers}) of $G$, which broadly speaking measure how equivariant homotopy theory interacts with chromatic homotopy theory. They are able to deduce these numbers for groups of square-free order from the seminal blue-shift result of Kuhn~\cite{kuhn}, and propose a conjecture for the general case. Here, we resolve this conjecture for all finite {\em abelian} groups $A$, thereby giving a complete description of $\Spc(\sp_A^\omega)$, and we deduce geometric consequences for finite complexes with $A$-action (see Theorem \ref{thm:main}). In particular, this establishes a far-reaching generalization and geometric interpretation of Kuhn's theorem. 

\subsection*{Detailed description of the results}


We now go through the results of this paper in more detail and recall the key definitions along the way. Let $(\cT, \otimes, \bf{1})$ be an essentially small $\otimes$-triangulated category \cite{verdier, neeman}. The points of the Balmer spectrum $\Spc(\cT)$ of $\cT$ are the prime thick $\otimes$-ideals, i.e., those proper thick $\otimes$-ideals $\cI\subsetneq\cT$ such that if $a\otimes b\in \cI$, then $a\in \cI$ or $b\in \cI$. A basis for the open subsets of $\Spc(\cT)$ is given by the complements of subsets of the form $\mathrm{Supp}(a)=\{\cP\in \Spc(\cT)\, \mid \, a\not\in \cP\}$, for some $a\in \cT$. When $\cT$ is rigid\footnote{The consequence of this technical assumption is that every thick $\otimes$-ideal of $\cT$ is radical, see \cite[Rem.~4.3 and Prop.~4.4]{balmer-the-spectrum}.}, and 
it will be in all examples below, then Balmer's classification theorem \cite[Introduction]{balmer-the-spectrum} identifies the collection of thick $\otimes$-ideals of $\cT$ with the Thomason subsets of $\Spc(\cT)$, i.e., subsets given by unions of closed subsets with quasi-compact complements. 
This shows in particular that the underlying set of $\Spc(\cT)$ determines only the prime thick $\otimes$-ideals, while the topology of $\Spc(\cT)$ is required to classify all thick $\otimes$-ideals of~$\cT$.

In \cite{Balmer-Sanders}, Balmer and Sanders study, for a finite group $G$, the Balmer spectrum $\Spc(\sp_G^\omega)$  of the homotopy category $\sp_G^\omega$ of compact genuine $G$-spectra \cite{LMS86}. 
We recommend the introduction of \cite{Balmer-Sanders} for a thorough overview of this problem. The results depend on the thick subcategory theorem of Hopkins and Smith \cite{Hopkins-Smith,ravenel-orange}, which we will now recall (see \cite[Sec.~9]{spectra3} for more details)\footnote{Note that when the tensor unit $\bf{1}\in\cT$ generates $\cT$ as a thick subcategory, then every thick subcategory of $\cT$ is also a thick $\otimes$-ideal. So there is no distinction between thick $\otimes$-ideals and thick subcategories of $\sp^\omega$.}.
For each prime $p\in \bZ$ and integer $n\ge 1$, there is a prime thick $\otimes$-ideal in the $\infty$-category $\sp^\omega$ of finite spectra
\[ C^n_p:=\{ X\in \sp^\omega \,\mid\, K(n-1)_*X=0 \},\]
and these constitute a descending chain
\[C^1_p\supseteq\ldots\supseteq C^n_p\supseteq\ldots\supseteq C_p^\infty:=\bigcap_{n\ge 0} C^n_p=\{ X\in \sp^\omega\,\mid\, K(\infty)_*(X)=0\} .\]
Here, $K(n)$ denotes the $n$th Morava $K$-theory at the prime $p$ with the usual
conventions that $K(0)=H\mathbb{Q}$ (independently of $p$) and $K(\infty)=H\mathbb{F}_p$. Finally, $\Spc(\sp^{\omega})$ is obtained by taking the union over all $p$ of these sets of prime ideals and noting that, independently of $p$, each of the prime ideals $C_p^1$ equals the single prime consisting of torsion finite spectra.

Now we return to $\sp_G^\omega$. For each subgroup $H$ of $G$ we have an exact symmetric monoidal geometric fixed point functor \[ \Phi^{H}\colon \sp_{G}^\omega\longrightarrow \sp^\omega \] which induces a continuous map of Balmer spectra ${\Phi^{H}}^*\colon \Spc(\sp^\omega)\to \Spc(\sp_G^\omega)$.  Balmer and Sanders show that these maps are jointly surjective and that ${\Phi^{H_1}}^*(\mathfrak{p})={\Phi^{H_2}}^*(\mathfrak{q})$ if and only if $H_1$ is conjugate to $H_2$ and $\mathfrak{p}=\mathfrak{q}$ in $\Spc(\sp^\omega)$ \cite[Thm.~4.9 and Thm.~4.11]{Balmer-Sanders}. This determines $\Spc(\sp_G^\omega)$ as a set. To complete the identification of the {\em topological space} $\Spc(\sp_G^\omega)$, and hence obtain the classification of the thick $\otimes$-ideals, one needs to further identify all inclusions between the prime ideals
\[ \Spc(\sp_G^\omega) =  \{ {\mathcal P}(H,q,n):=(\Phi^H)^{-1}(C^n_q) \mid H\subseteq G, 1\leq n\leq \infty, q \textrm{ prime}\} . \footnote{The fact that determining these inclusions is {\em equivalent} to knowing the topology follows from the second sentence of \cite[Cor.~8.19]{Balmer-Sanders}. }
\] 
Balmer and Sanders reduce this problem to the special case that $G$ is a $p$-group, for some prime $p$, and $q=p$ \cite[Prop.~6.11]{Balmer-Sanders}. They also give an important result in this case \cite[Prop.~8.1, Cor.~8.4]{Balmer-Sanders}: Suppose that $G$ is a $p$-group, $K\subseteq H$ are subgroups of $G$, and denote $s:=\log_p(|H/K|)$. Then, for each $n\ge 1$, \[{\mathcal P}(K,p,n+s) \subseteq {\mathcal P}(H,p,n).\]
After the other reductions of Balmer and Sanders, the only remaining question is as follows, cf.~\cite[Rem.~8.6]{Balmer-Sanders}.

\begin{question}\label{quest} 
Let $K\subseteq H$ be subgroups of a $p$-group $G$ and $1\leq n<\infty$. What is the minimal $0\leq i\leq \log_p(|H/K|)$ such that ${\mathcal P}(K,p,n+i)\subseteq {\mathcal P}(H,p,n)$?
\end{question}

To underline their fundamental importance, we give these numerical invariants of finite $p$-groups a proper name.

\begin{definition}\label{def:blue-shift-numbers} 
In the situation of \Cref{quest},
let $\beth_n(G;H,K):=i$ and call it {\em the $n$th blue-shift number of G with respect to $K\subseteq H$}.
\end{definition}

The term blue-shift here is motivated by the fact that 
Balmer and Sanders link \Cref{quest}, which asks how the chromatic type of the geometric fixed points of a finite $G$-spectrum can vary, to a blue-shift phenomenon for generalized Tate cohomology, cf. \cite[Section 9]{Balmer-Sanders}. 

In summary, computing, for a $p$-group $G$, all of the blue-shift numbers $\beth_n(G;H,K)$ is equivalent to identifying the topological space $\Spc(\sp_G^{\omega})$. Moreover, the identification of the Balmer spectrum for general finite groups reduces to the case of $p$-groups. Balmer and Sanders conjecture \cite[$\log_p$-Conjecture 8.7]{Balmer-Sanders} that their bound is optimal, namely that $\beth_n(G;H,K) = \log_p(| H/K |)$. In \Cref{cor:include} we will identify $\beth_n(G;H,K)$ precisely when $G=A$ is an {\em abelian} $p$-group. Our identification agrees with the Balmer--Sanders conjecture when $A$ is an elementary abelian $p$-group, and disagrees otherwise. We then provide the necessary correction to complete the identification of $\Spc(\sp_A^\omega)$ for all abelian groups~$A$.

To make \Cref{quest} more concrete, recall that Hopkins and Smith define for a finite $p$-local spectrum $Y\in C_p^0:=\sp^\omega_{(p)}$, $\mathrm{type}(Y):=\max\{ n\,\mid\, Y\in C^n_p\}\in [0,\ldots,\infty]$ to be {\em the type of $Y$}; here we abuse notation and let $C^n_p$ also denote the $p$-local analogues of the subcategories discussed above. Using the reduction to the $p$-local case and unwinding the definitions, one sees that \Cref{quest} is concerned with determining how the type of $\Phi^H(X)$ varies with respect to a choice of subgroup $H\subseteq G$, for a finite $p$-local $G$-spectrum $X$; see \Cref{rem:akhil} for a simple example along these lines. 
\tabularnewline

We now fix a prime $p$, a finite abelian group $A$ (not necessarily a $p$-group), and we let $\Sigma$ denote the set of subgroups of $A$. To every $X\in \sp_{A,(p)}^\omega$ we associate the function
$f_X\colon \Sigma\longrightarrow [0,1,\ldots,\infty]$, defined by $f_X(A'):=\mathrm{type}(\Phi^{A'}(X))$. This function encodes which prime thick $\otimes$-ideals of $\sp_{A,(p)}^\omega$ the spectrum $X$ belongs to and, as we vary $X$, the inclusions among all prime thick $\otimes$-ideals. 
Finally, let us denote by $\prank(B):=\dim_{\mathbb{F}_p} (B\otimes_{\mathbb{Z}}\mathbb{F}_p)$ the {\em $p$-rank} of a finite abelian group $B$.
\begin{thm}\label{thm:main}
	For a function $f\colon\Sigma\to [0,1,\ldots,\infty]$, the following are equivalent:
	\begin{enumerate}[label = \roman*)]
		\item \label{it:main-fn} There is some $X\in\sp_{A,(p)}^\omega$ such that $f=f_X$.
		\item \label{it:main-chain} For every chain of subgroups $A'\subseteq A''\subseteq A$ such that $A''/A'$ is a $p$-group, we have \[f(A') \leq f(A'')	+\prank(A''/A'). \]
	\end{enumerate}
\end{thm}

This answers \Cref{quest} for abelian $p$-groups, as follows.

\begin{corollary}\label{cor:include}
When $p$ is a prime, $A$ is an abelian $p$-group, and $K\subseteq H\subseteq A$ are subgroups, then $\beth_n(A;H,K)=\prank(H/K)$ for all $1\leq n <\infty$.
\end{corollary}

We note that as an immediate consequence of \cref{cor:include}, the blue-shift numbers $\beth_n(A;H,K)$ are independent of $n$. In case $A=\mathbb{Z}/p\mathbb{Z}$, \Cref{thm:main} is due to Balmer and Sanders \cite[Sec.~7]{Balmer-Sanders}.

To prove it, they make use of Kuhn's seminal blue-shift theorem for Tate cohomology \cite[Prop.~1.11]{kuhn}. We will prove a generalization of Kuhn's theorem, which also is of interest in its own right.  
We must first fix some notation. For a spectrum $X\in\sp$, let $\varphi^A(X):=\Phi^A(\underline{X})\in\sp$ be the $A$-geometric fixed points of the Borel $A$-equivariant spectrum $\underline{X}$ associated with $X$\footnote{For comparison with \cite{kuhn}, observe that $\varphi^A(X)=\Phi^A(\tilde{E}A\wedge F(EA_+, i_*X))$ are the {\em geometric} fixed points of the Tate-construction of the $A$-spectrum $i_*X$, where $i_* X$ is the inflation of $X$.}. Moreover, let $L_{n-1}^f$ denote the Bousfield localization functor on $\sp^\omega$ with kernel $C^{n}_p$. 
\begin{thm}\label{thm:kuhnplus}
	For every abelian $p$-group $A$ and integer $n\ge 1$ we have
	\[ \{ X\in C^0_p\,\mid\,\Phi^A(L_{n-1}^f(S^0))\otimes X=0\}=C_p^{\max( n-\prank(A), 0 ) } .\]

\end{thm}


\subsection*{Outline of the proof of \cref{thm:main}}
We first show that for abelian groups $A$, the blue-shift numbers $\beth_n(A;A',A'')$
are determined by the absolute ones $\beth_n(A):=\beth_n(A;A,e)$ (\cref{lem:going-up} and \cref{lem:inflate}).
In terms of the $\beth_n(A)$, the proof of \cref{thm:main} has two major ingredients which establish a lower and an upper bound, respectively. 

The first ingredient is a far-reaching generalization of the blue-shift theorem for complex oriented cohomology theories due to Greenlees, Hovey, and Sadofsky~\cite{Greenlees-Sadofsky, Hovey-Sadofsky}. Indeed, the natural map from $L_{n-1}^fS^0$ to a Morava $E$-theory spectrum $E$ of height $n-1$ allows us to obtain an upper bound for $\beth_n(A)$ by considering $\varphi^A(E)$. The key idea is then to recognize $\varphi^{{A}}(E)$ as suitable sections of the structure sheaf on a certain non-connective derived scheme. This makes it possible to use derived algebraic geometry and the geometry of the stack of formal groups to completely describe the height shifting behaviour of $\varphi^{{A}}(E)$. 

As a second ingredient, we make use of equivariant finite complexes $F(n)$ built from partition complexes that feature prominently in the Goodwillie calculus of functors. In particular, we rely on the study of their chromatic behaviour in~\cite{ArM99, Aro98} and on that of their fixed points in~\cite{ArL17}. These results imply that the complexes $F(n)$ realize the stipulated blue-shift, thereby establishing our lower bound for $\beth_n(A)$.

Finally, we show that our upper and lower bounds for $\beth_n(A)$ coincide and deduce the description of $\Spc(\sp_A^\omega)$. 
Since for non-abelian groups $G$ one has $\varphi^G(E)=0$ (\cite[Prop.~5.26]{MNN}), and there are no suitable non-abelian generalizations of the $F(n)$ known, we believe that the determination of the Balmer spectrum for any interesting class of finite non-abelian groups requires substantial new ideas. In fact, we are not even able to see that for a general finite
group $G$, the number $\beth_n(G)$ is independent of $n$.

\subsection*{Organization of the paper} This paper is organized as follows: In \Cref{sec:proofs} we prove \cref{thm:main} and \cref{thm:kuhnplus}, postponing the proofs of two key technical results to the following two sections. The first technical result is \Cref{nthm:heightdrop} in \Cref{subset:blue-on-lubin-tate} which determines the blue-shift of generalized Tate-constructions on Lubin--Tate spectra.
It will be used to establish the implication \ref{it:main-fn} $\Rightarrow$ \ref{it:main-chain} in \Cref{thm:main}. The second one is \Cref{prop:complexes-restate} in \Cref{subsec:arone-lesh}, and is a guide through 
previous work of Arone, Lesh, Dwyer, and Mahowald which provides examples of finite $A$-complexes with very subtle properties. This is the key result needed to show the implication \ref{it:main-chain} $\Rightarrow$ \ref{it:main-fn} in \Cref{thm:main}.

\subsection*{Acknowledgments}
Markus Hausmann thanks Gregory Arone for helpful conversations on the subject of this paper. Niko Naumann thanks Neil Strickland for making unpublished notes of his available. We thank Paul Balmer, Beren Sanders, and the anonymous referee for pointing out inaccuracies in a preliminary draft of this paper. This work first began while Tobias Barthel, Thomas Nikolaus, and Nathaniel Stapleton were in Bonn and they thank the MPIM for its hospitality.

\section{Proofs of the main results}\label{sec:proofs}

We first state two key technical results, the proofs of which are postponed to
\Cref{subset:blue-on-lubin-tate} and \Cref{subsec:arone-lesh}, respectively.

\begin{thm}\label{cor:bound-height-drop}
	Assume $p$ is a prime, $A$ is a finite abelian $p$-group and $X\in\sp_{A,(p)}^\omega$. Then type$(\Phi^A(X))\ge\mathrm{type}(\Phi^{\{0\}}(X))-\prank(A)$. 
\end{thm}

\begin{thm}\label{prop:complexes}(Arone, Dwyer, Lesh, Mahowald)\\
Let $p$ be a prime, $n\ge 1$ and $\Delta:=(\mathbb{Z}/p\mathbb{Z})^{\times n}$ the corresponding elementary abelian $p$-group. Then, there is a $p$-local finite $\Delta$-equivariant spectrum $F(n)\in\sp_{\Delta,(p)}^\omega$ satisfying the following conditions:
	\begin{enumerate}[label = \roman*)]
		\item\label{it:adlm-0} The geometric fixed points $\Phi^{\Delta}(F(n))$ have type $0$.
		\item\label{it:adlm-n} The underlying non-equivariant spectrum of $F(n)$, i.e., $\Phi^{\left\{ 0\right\}}(F(n))$, has type $n$.
	\end{enumerate}
\end{thm}

Fix a finite abelian group $A$ and a prime $p$ (where $A$ is not required to be a $p$-group). Now we turn to determining the Balmer spectrum of $\sp_{A,(p)}^\omega$, the localization at $p$ of $\sp^\omega$. For every subgroup $A'\subseteq A$ and 
$1\leq n\leq\infty$, we have a prime ideal
\[ {\mathcal P}(A',n):= {\mathcal P}(A',p,n):={\Phi^{(A')}}^{-1}(C_p^n)= \{ X\in \sp_{A,(p)}^\omega\,\mid\, \mathrm{type}(\Phi^{A'}(X))\ge n\}\in \mathrm{Spc}(\sp_{A,(p)}^\omega)\footnote{Since we are now working $p$-locally, we will omit the second entry in the notation ${\mathcal P}(A',q,n)$ from the introduction, since $q$ will always be $p$.}. \]
Writing $\Sigma$ for the set of subgroups of $A$, even more is true: The map
\[ \Sigma\times [1,\ldots, \infty] \longrightarrow \mathrm{Spc}(\sp_{A,(p)}^\omega)\, , \, (A',n)\mapsto {\mathcal P}(A',n) \]
is bijective \cite[Thms.~4.9 and 4.14]{Balmer-Sanders}. Determining the topology on $\mathrm{Spc}(\sp_{A,(p)}^\omega)$ is more subtle, and is equivalent to deciding for which pairs $(A',n),(A'',m)\in \Sigma\times [1,\ldots, \infty]$ we have an inclusion ${\mathcal P}(A',n)\subseteq {\mathcal P}(A'',m)$, cf.~\cite[Cor.~8.19]{Balmer-Sanders}.

Our result is as follows.

\begin{thm} \label{thm:topology}
	Given subgroups $A',A''\subseteq A$ and $1\leq n,m\leq \infty$, the following are equivalent:
	\begin{enumerate}[label = \roman*)]
		\item \label{it:top-contain} We have ${\mathcal P}(A',n)\subseteq {\mathcal P}(A'',m)$.
		\item \label{it:top-prank} We have $A'\subseteq A''$, the quotient $A''/A'$ is a $p$-group, and $n\ge m+\prank(A''/A')$.
	\end{enumerate}
\end{thm}
We will now explain how to use results of \cite{Balmer-Sanders} to reduce the proof of \Cref{thm:topology} to the special case that $A'\subseteq A''$ equals
$\{ 0 \} \subseteq A$:
By \cite[Prop.~6.9]{Balmer-Sanders}, the inclusion in \Cref{thm:topology}, \ref{it:top-contain} is possible only 
if $A'\subseteq A''$ and $A''/A'$ is a (possibly trivial) $p$-group, so we assume this from now on. In the following, we will write ${\mathcal P}_A(?,?) = \mathcal{P}(?,?)$ when it seems necessary to identify that the ambient group is $A$. We can then state the first reduction step:

\begin{prop}\label{lem:going-up}
	For subgroups $A'\subseteq A''\subseteq A$ and $1\leq m,n\leq\infty$, the following are equivalent:
	\begin{enumerate}[label = \roman*)]
		\item\label{it:gup-ambient} We have ${\mathcal P}_A(A',n)\subseteq {\mathcal P}_A(A'',m)$.
		\item\label{it:gup-subgp} We have ${\mathcal P}_{A''}(A',n)\subseteq {\mathcal P}_{A''}(A'',m)$.
	\end{enumerate}
\end{prop}

\begin{proof}
	Restriction induces a continuous map of spectra \[ \mathrm{Res} \colon  \mathrm{Spc}(\sp_{A'',(p)}^\omega)\longrightarrow \mathrm{Spc}(\sp_{A,(p)}^\omega) \]
	which, for all subgroups $B\subseteq A''$ and $1\leq k\leq\infty$, satisfies $\mathrm{Res}({\mathcal P}_{A''}(B,k))= 
	{\mathcal P}_A(B,k)$. So \ref{it:gup-subgp} implies \ref{it:gup-ambient}.

	The reverse implication is more subtle and uses the observation \cite[Sec.~2.1, (F)]{Balmer-Sanders} that Res
	satisfies a Going-Up theorem. Following the notation there, choose ${\mathcal P}':={\mathcal P}_{A}(A',n)$
	and ${\mathcal Q}:={\mathcal P}_{A''}(A'',m)$. We then have Res$({\mathcal Q})={\mathcal P}_{A}(A'',m)
	\supseteq {\mathcal P}'$, i.e., ${\mathcal P}'\in\overline{\{ \mathrm{Res}({\mathcal Q}) \}}$, and the Going-Up theorem
	implies that there is some ${\mathcal Q}'\in\overline{\{ {\mathcal Q} \}}$ such that Res$({\mathcal Q}')={\mathcal P}'$.
	Since for subgroups of abelian groups, Res is injective \cite[Cor.~4.4 and Thm.~4.14]{Balmer-Sanders} and we have Res$({\mathcal P}_{A''}(A',n))={\mathcal P}_{A}(A',n)={\mathcal P}'$, we conclude that ${\mathcal Q}'={\mathcal P}_{A''}(A',n)$, and the relation
	${\mathcal Q}'\in\overline{\{ {\mathcal Q} \}}$ then means that \ref{it:gup-subgp} holds.
\end{proof}
 Replacing $A$ by $A''$, we have thus reduced to understanding possible inclusions ${\mathcal P}_A(A',n)\subseteq  {\mathcal P}_A(A,m)$. We reduce this problem further with the following proposition.

\begin{prop}\label{lem:inflate}
	For a subgroup $A'\subseteq A$ and $1\leq m,n\leq\infty$, the following are equivalent:
	\begin{enumerate}[label = \roman*)]
		\item We have ${\mathcal P}_A(A',n)\subseteq {\mathcal P}_A(A,m)$.
		\item We have ${\mathcal P}_{A/A'}(\{ 0\},n)\subseteq {\mathcal P}_{A/A'}(A/A',m)$.
	\end{enumerate}
\end{prop}
 
\begin{proof}
	We use \cite[\S 2.1, (H)]{Balmer-Sanders}. The adjoint functors \[ \Phi^{A'}\colon \sp_A\leftrightarrow\sp_{A/A'}\colon \mathrm{Inf}\]
	satisfy $\Phi^{A'}\circ\mathrm{Inf}\simeq\mathrm{id}$. Hence the induced maps on spectra exhibit $\mathrm{Spc}(\sp_{A/A',(p)}^\omega)$ as a retractive subspace of $\mathrm{Spc}(\sp_{A,(p)}^\omega)$. Furthermore, by
	\cite[Prop.~4.7]{Balmer-Sanders}, we have $\mathrm{Spc}(\Phi^{A'})({\mathcal P}_{A/A'}(\{ 0\},n))={\mathcal P}_{A}(A',n)$ and
	$\mathrm{Spc}(\Phi^{A'})({\mathcal P}_{A/A'}(A/A',m))={\mathcal P}_{A}(A,m)$, which concludes the proof.
\end{proof}

To summarize, for the proof of \Cref{thm:topology} we can assume $A'\subseteq A''$ with $A''/A'$ a $p$-group, by the discussion immediately following \Cref{thm:topology}, we can then assume that $A=A''$ by \Cref{lem:going-up}, and that $A'=0$ by \Cref{lem:inflate}. This now reduces the proof of \Cref{thm:topology} to the following special case.
\begin{thm}\label{thm:key-lemma}
	Assume $A$ is a finite abelian $p$-group and $1\leq m,n\leq\infty$. 
	Then the following are equivalent:
	\begin{enumerate}[label = \roman*)]
		\item \label{it:key-sub} We have ${\mathcal P}(\{ 0 \},n)\subseteq {\mathcal P}(A,m)$.
		\item \label{it:key-prank} We have $n\ge m+\prank(A)$.
	\end{enumerate}
\end{thm}

\begin{proof}
	Let $k:=\prank(A)$.
	We first assume that $m,n<\infty$.
	To show that \ref{it:key-sub} implies \ref{it:key-prank}, we will prove the contrapositive. In other words, we assume that $m>n-k$, and then we will show that 
	${\mathcal P}(\{ 0\},n)\not\subseteq {\mathcal P}(A,m)$. To do this, we consider two cases. 

	First assume that $n\ge k$. Since $m\ge n-k+1\ge 1$, we have
	${\mathcal P}(A,m)\subseteq {\mathcal P}(A,n-k+1)$, and it suffices to show that
	${\mathcal P}(\{ 0\},n)\not\subseteq {\mathcal P}(A,n-k+1)$, i.e., that there is some
	$X\in\sp_{A,(p)}^\omega$ with type$(\Phi^{\{ 0 \}}(X))\ge n$ and type$(\Phi^A(X))\leq n-k$.
	Our example will achieve equality in both cases. To start, use \Cref{prop:complexes}
	to choose some $Y\in\sp_{(\mathbb{Z}/p\mathbb{Z})^{\times n},(p)}^\omega$ with type$(\Phi^{\{ 0 \}}(Y))=n$ and
	type$(\Phi^{(\mathbb{Z}/p\mathbb{Z})^{\times n}}(Y))=0$. We choose a subgroup $(\mathbb{Z}/p\mathbb{Z})^{\times k}\subseteq (\mathbb{Z}/p\mathbb{Z})^{\times n}$,
	and consider $Z:=\mathrm{Res}^{(\mathbb{Z}/p\mathbb{Z})^{\times n}}_{(\mathbb{Z}/p\mathbb{Z})^{\times k}}(Y)\in\sp_{(\mathbb{Z}/p\mathbb{Z})^{\times k},(p)}^\omega$.
	It satisfies type$(\Phi^{\{ 0 \}}(Z))=n$, and the type of $\Phi^{(\mathbb{Z}/p\mathbb{Z})^{\times n}}(Y)\simeq \Phi^{(\mathbb{Z}/p\mathbb{Z})^{\times n}/(\mathbb{Z}/p\mathbb{Z})^{\times k}}\left( \Phi^{(\mathbb{Z}/p\mathbb{Z})^{\times k}}(Z) \right)$ is zero. Since the individual fixed-point functors can drop height at most by the rank of the group involved by \Cref{cor:bound-height-drop},
	we conclude that type$(\Phi^{(\mathbb{Z}/p\mathbb{Z})^{\times k}}(Z))=n-k$. Denoting by $X$ the inflation of 
	$Z$ along $A\to A/pA\simeq (\mathbb{Z}/p\mathbb{Z})^{\times k}$, it is now clear that $X$ has the desired properties. 

	In the case $n<k$, it suffices to check that ${\mathcal P}(\{ 0\},k) \not\subseteq {\mathcal P}(A,m)$, i.e., that there exits $X\in\sp_{A,(p)}^\omega$ with type$(\Phi^{\{ 0 \}}(X))\ge k$ and type$(\Phi^A(X))\leq m-1$. The existence of such a spectrum follows again from \Cref{prop:complexes}, since $k$ is the $p$-rank of $A$, and hence we can even
	find some $X\in\sp_{A,(p)}^\omega$ with type$(\Phi^{\{0\}}(X))=k$ and type$(\Phi^A(X))=0$.

	To see that \ref{it:key-prank} implies \ref{it:key-sub}, we need to see that if $X\in\sp_{A,(p)}^\omega$ satisfies
	type$(\Phi^{\{ 0 \}}(X))\ge n\geq m+k$, then type$(\Phi^A(X))\ge m$. This is precisely the content of \Cref{cor:bound-height-drop}.

	The remaining cases with $\infty\in\{ m,n\}$ reduce to showing that ${\mathcal P}(\{ 0 \} , \infty)\subseteq
	{\mathcal P}(A,\infty)$ which follows from the above as in \cite[Cor.~7.2]{Balmer-Sanders}.
	\end{proof}

\begin{proof}[Proof of \Cref{thm:main}]
	To see that \ref{it:main-fn} implies \ref{it:main-chain}, take subgroups $A'\subseteq A''\subseteq A$ such that $A''/A'$ is a $p$-group, say of $p$-rank $k$, and also fix some $X \in \sp_{A,(p)}^\omega$. We wish to show that 
	\[ f(A'')=\mathrm{type} (\Phi^{A''}(X))\geq \mathrm{type}(\Phi^{A'}(X))-k\,\,\, (=f(A')-k). \]
	Restricting to $A''$ we can assume that $A=A''$, and using that in this case $\Phi^A=\Phi^{A/A'}\circ\Phi^{A'}$, we can further reduce to $A'=\{ 0\}$. In this case, the desired conclusion is \Cref{cor:bound-height-drop}.

	To see that \ref{it:main-chain} implies \ref{it:main-fn}, consider the following subset
	\[ \mathrm{Spc}(\sp_{A,(p)}^\omega)\supseteq{\mathcal Z}:=\{ {\mathcal P}(A',n)\, \mid\, 
	\forall A'\subseteq A : n > f(A')\} \]
	(where, in the above expression, we use the convention $\infty\ngtr\infty$). This subset is closed by our assumption \ref{it:main-chain} on the function $f$ and \Cref{thm:topology}. Indeed, if ${\mathcal P}(A',n) \in \mathcal{Z}$ and ${\mathcal P}(B,l) \subseteq  {\mathcal P}(A',n)$ (i.e., ${\mathcal P}(B,l)\in\overline{{\mathcal P}(A',n)}$), then 
	\[
	l \ge n+ \prank(A'/B) > f(A') + \prank(A'/B) \ge f(B),
	\]
hence ${\mathcal P}(B,l) \in \mathcal{Z}$. The open complement of $\mathcal Z$ is quasi-compact by \cite[Prop.~10.1]{Balmer-Sanders}. By \cite[Prop.~2.14]{balmer-the-spectrum}, there is some $X\in\sp_{A,(p)}^\omega$ with supp$(X)={\mathcal Z}$. It is clear that this $X$ has the desired properties.
\end{proof}

\begin{proof}[Proof of \Cref{thm:kuhnplus}]
	Fix an integer $n\ge 1$, set $k:=\prank(A)$, and recall we wish to prove that 
\[ \{ X\in C^0_p\,\mid\, \Phi^A( L_{n-1}^f(S^0))\otimes X=0\}=C_p^{\max( n-\prank(A), 0 ) } .\]
	To show the inclusion
	$\subseteq$, we can assume that $n-k\ge 1$, for otherwise the claim is trivial.
	Take $X\in C_p^0$ with $\varphi^A(L_{n-1}^f (S^0))\otimes X\simeq *$.
	Since there is a ring map $L_{n-1}^f S^0\longrightarrow E$ for a Lubin--Tate
	theory $E$ at $p$ of height $n-1$, we see that 
	$\varphi^A(E)\otimes X\simeq *$ as well. Since the chromatic height of
	$\varphi^A(E)$ is $n-1-k$ (\Cref{rem:height-drop}, \ref{it:geom-ht-drop}), we have $X\in C_p^{n-k}$ by \Cref{nlem:heightcheck}, \ref{it:hc-ker}.

	To see the inclusion $\supseteq$, by the thick subcategory theorem it is sufficient
	to find a single example of some $X\in C_p^{\max( n-k, 0 ) }$ such that $\varphi^A(L_{n-1}^f(S^0))\otimes X\simeq *$. As in the proof 
	of \Cref{thm:key-lemma}, we see that there is some $Y\in\sp_{A,(p)}^\omega$
	such that type$(\Phi^A(Y))=\max( n-k,0)$ and such that type$(\Phi^{\{ 0 \}}Y)=\max (n,k)$.
	We then have $L_{n-1}^f(S^0)\otimes Y\simeq *$, and hence $\underline{L_{n-1}^f(S^0)\otimes Y}\simeq \underline{L_{n-1}^f(S^0)}\otimes Y\simeq *$, which implies:
	\[ *\simeq \varphi^A(L_{n-1}^f(S^0))\otimes \Phi^A(Y),\]
	and $X:=\Phi^A(Y)$ is as desired.
\end{proof}

\section{Blue-shift for Lubin--Tate spectra}\label{subset:blue-on-lubin-tate}

The aim of this subsection is to compute the blue-shift on Lubin--Tate spectra of ${\mathcal F}$-geometric fixed points
for general families ${\mathcal F}$ in abelian groups, see \Cref{nthm:heightdrop} below. There are many similar and overlapping results in the literature,
and while we will not try to be exhaustive here, we mention at least the following sources:  \cite{Greenlees-Sadofsky} consider the Tate cohomology of $v_n$-periodic complex oriented spectra, \cite{Hovey-Sadofsky} consider the Bousfield classes of the Tate cohomology of the $L_n$-localizations of finite spectra, \cite{stroilova-phd} considers elementary abelian $p$-groups and an alternative Tate construction, \cite{ando-morava-sadofsky} consider the $\mathbb{Z}/p\mathbb{Z}$-Tate construction on Johnson--Wilson spectra, and \cite{hkr} provide crucial results on the Lubin--Tate cohomology of finite groups which we will use below.

Fix a prime $p$, an integer $n\ge 1$, and a Lubin--Tate spectrum $E$ of height $n$
at the prime $p$, see \cite{Rez97,LurieChromaticCourse} for general background.
We denote by $L_n:=L_E$ the corresponding Bousfield localization functor, cf. \cite[Ch.~7]{ravenel-orange}. This localization only depends on $n$ (and the implicit prime $p$), and not on our choice of Lubin--Tate spectrum $E$.

\begin{definition} 
	The {\em chromatic height} of an $\mathbb{E}_\infty$-$E$-algebra $E\to R\not \simeq 0$
	is \[ \mathrm{ht}(R):=\min\{ t\ge 0\,\mid\, R\xrightarrow{\simeq} L_tR\}.\]
\end{definition}

\begin{remark} 
	Since $E\simeq L_nE$, we have $0\leq\mathrm{ht}(R)\leq n$. By analyzing chromatic fracture squares (cf. \cite[Eq.~(0.1)]{goerss-res}, one sees that $\mathrm{ht}(R)=t$ is equivalent to $K(i)_*R=0$ for $i>t$ and $K(t)_*R\neq 0$. By \cite[Thm.~1.1]{hahnheight}, the former condition follows from only knowing $K(t+1)_*R=0$, so we see that 
	$\mathrm{ht}(R)=\max\{t\ge 0\, \mid \, K(t)_*R\neq 0\}$.
\end{remark} 

Fix a finite abelian $p$-group $A$.

\begin{definition}\label{ndef:prank}\ 
	\begin{enumerate}[label = \roman*)]
		\item For a proper family ${\mathcal F}$ of subgroups of $A$, we call 
		\[ \coprank({\mathcal F}):=\min\{\prank(A')\,\mid\, A'\subseteq A\mbox{ such that }
		A'\not\in{\mathcal F} \}\]  the {\em $p$-corank of ${\mathcal F}$}.

		\item\label{it:families} For a family ${\mathcal F}$ of subgroups of $A$, there are $A$-spaces $E\cF$ and $\wt{E}\cF$, the latter of which is pointed, which are characterized, up to an essentially unique equivariant weak equivalence, by their fixed point data:
		 \begin{equation}
E\cF^K  \simeq 
	\begin{cases} 
		* & \mbox{if } K\in \cF\\
		\emptyset & \mbox{otherwise}
	\end{cases}\label{eq:EF-univ-prop}
\quad \quad \quad \wt{E}\cF^K  \simeq 
	\begin{cases} 
		* & \mbox{if } K\in \cF\\
		S^0 & \mbox{otherwise.}
	\end{cases}	
\end{equation}
		\item In \ref{it:families}, when $\cF=\{\{ 0\}\}$ is the family only containing the trivial subgroup, it is customary to write $EA:=E\cF$ and $\wt{E}A:=\wt{E}\cF$. 
		\item For a spectrum $X\in \sp$, let $\underline{X} \in \sp_A $ denote the \emph{Borel completion} of $X$, that is the unique $A$-spectrum which is Borel complete  and whose underlying spectrum is $X$ with trivial $A$-action, cf.~\cite[Sec.~6.3]{MNN17}.
		 		\item For a family ${\mathcal F}$ of subgroups of $A$, and $X\in \sp_A$, we call
		\[ \Phi^{{\mathcal F}}(X):=\left(\wt{E}{\mathcal F}\otimes X\right)^A\] the {\em ${\mathcal F}$-geometric fixed points of $X$}. As special cases, $\Phi^{\{ 0\}} (\underline{X})=t^A(X)$ is the classical Tate construction as in \cite{greenlees-may-tate}, and for the family ${\mathcal P}$ of proper subgroups of $A$, $\Phi^A(X):=\Phi^{{\mathcal P}}(X)$ are the (usual)
		$A$-geometric fixed points.
					\end{enumerate}
\end{definition}

The following is the main result of this section.

\begin{thm}\label{nthm:heightdrop}
Let $p$ be a prime, $A$ a finite abelian $p$-group, $n\ge 1$ an integer, $E$
a Lubin--Tate spectrum at $p$ of height $n$ and ${\mathcal F}$ a family of subgroups of $A$.
	Then $\Phi^{{\mathcal F}}(\underline{E})=0$
	if and only if $\coprank({\mathcal F})\ge n+1$. Otherwise, the chromatic height of the $\mathbb{E}_\infty$-$E$-algebra $\Phi^{{\mathcal F}}(\underline{E})$ is given by
	\[ \mathrm{ht}(\Phi^{{\mathcal F}}(\underline{E}))=\mathrm{ht}(E) - \coprank({\mathcal F}) = n - \coprank({\mathcal F}).\]
\end{thm}

\begin{rem}\label{rem:height-drop}\ 

	\begin{enumerate}[label = \roman*)]
		\item In the language of \cite{MNN}, the vanishing criterion in \Cref{nthm:heightdrop} for $\Phi^{{\mathcal F}}(\underline{E})$ is equivalent to the determination of the 
		derived defect base of $\underline{E}$: Since $\wt{E}{\mathcal F}\otimes
		\underline{E}$ is a ring spectrum, its $A$-fixed points, i.e., $\Phi^{{\mathcal F}}(\underline{E})$, vanish if and only if
		it is itself equivariantly contractible, i.e., $\wt{E}{\mathcal F}\otimes\underline{E}=0$. By definition, this is equivalent to
		${\mathcal F}$ containing the derived defect base of $\underline{E}$ which consists of those subgroups of $A$ of $p$-rank at 
		most $n$ by \cite[Prop.~5.36]{MNN}. \Cref{nthm:heightdrop} extends that result by further identifying how the chromatic height varies for all families of subgroups of $A$.

		\item \label{it:geom-ht-drop} For the family ${\mathcal P}$ of proper subgroups of a non-trivial finite abelian group $A$, the $p$-corank $\coprank({\mathcal P})=\prank(A)$ is
		the $p$-rank of $A$, hence the chromatic height of the {\em geometric} fixed points 
		$\Phi^A(\underline{E})=\Phi^{\mathcal P}(\underline{E})$ of 
		$\underline{E}$ drops by the $p$-rank of $A$. After taking into account \cite[Prop.~3.20]{greenlees_may_equivariant}, this case was implicitly studied in \cite[p.~1015]{Sta13}.
			\item If ${\mathcal F}$ is taken to be the family of subgroups 
		of $p$-rank at most $m<\prank(A)$, then the chromatic height drops by $\coprank({\mathcal F})=m+1$. In particular,
		every height drop between $0$ and rk$_p(A)$ can be realized by a suitable ${\mathcal F}$-geometric fixed points functor.
	\end{enumerate}
\end{rem}

The proof of \Cref{nthm:heightdrop} will be through a series of lemmas.
We first record a folklore result, \Cref{nlem:heightcheck} below, relating the height of $\mathbb{E}_\infty$-$E$-algebras with the geometry of
Lubin--Tate space (cf. \cite[Ch.~12, Lem. 8.1(2)]{behrens-tmf}).
To formulate it, we need to fix some notation first:
Choose a map of $\mathbb{E}_1$-algebras $MU_{(p)}\to
E$ and denote by $v_i\in\pi_{2(p^i-1)}(E)$ ($i\ge 0$, $v_0:=p$) the images
of the Araki-generators of the same name under this map. We can arrange that $v_n$ is a unit which admits a root $v_n=u_n^{1-p^n}$ (hence $u_n$ is of degree $-2$), and we then have $u_i:=v_i u_n^{p^i-1}\in \pi_0(E)$ ($0\leq i\leq n-1$). The deformation theory of formal groups implies that $\pi_*(E)=W(k)[[u_1,\ldots,u_{n-1}]][u_n^{\pm 1}]$
and that the $u_0, \ldots , u_{n-1}$ determine
the height filtration of the formal part of the universal $p$-divisible group ${\mathcal G}$ over $\mathrm{Spec}(\pi_0(E))$.
For every $0\leq t\leq n-1$, we denote by $I_{t+1}:=(u_0,\ldots u_t)\subseteq\pi_0(E)$ the ideal which
cuts out the locus of height at least $t+1$. We also set $I_{n+1}:=(1)$.

Recall (\cite[Def.~8.5, Thm.~8.42]{DAGVII}) that a non-connective spectral DM-stack can 
be thought of as a pair ${\mathfrak X}=(X,{\mathcal O}_{{\mathfrak X}})$ consisting of a 
classical Deligne-Mumford stack $(X,\mathcal{O}_X)$ (\cite{LMB}) and a hyper-complete sheaf ${\mathcal O}_{{\mathfrak X}}$ of $\mathbb{E}_\infty$-rings on the \'etale site of $X$ such that
$\pi_0({\mathcal O}_{{\mathfrak X}})\simeq{\mathcal O}_X$ and such that the ${\mathcal O}_X$-module $\pi_i({\mathcal O}_{{\mathfrak X}})$
is quasi-coherent for all $i\in\mathbb{Z}$. Every $\mathbb{E}_\infty$-ring $R$
canonically determines a non-connective spectral DM-stack $\mathrm{Sp\acute et}(R)=
(\mathrm{Spec}(\pi_0(R)), {\mathcal O}_{\mathrm{Sp\acute et}(R)})$  (cf. \cite[\S 1.4.2]{SAG})\footnote{The construction of ${\mathcal O}_{\mathrm{Sp\acute et}(R)}$ will partially be recalled during the proof of \Cref{nlem:section} below.}.

\begin{lemma}\label{nlem:heightcheck}
	Assume that $\emptyset\neq {\mathfrak X}=(X,{\mathcal O}_{{\mathfrak X}})\xrightarrow
	{f}\mathrm{Sp\acute et}(E)$ is a non-connective spectral DM-stack, and $0\leq t\leq n$ is an integer.
	Then the following are equivalent:
	\begin{enumerate}[label = \roman*)]
		 \item \label{it:hc-locsheaf} The sheaf of $\mathbb{E}_\infty$-rings ${\mathcal O}_{{\mathfrak X}}$ is $L_t$-local, i.e., for every \'etale open $U\to X$, the $\mathbb{E}_\infty$-ring ${\mathcal O}_{\mathfrak X}(U)$ is $L_t$-local.

		 \item \label{it:hc-cutout} The map of underlying spaces \cite[\S 1.5.4]{SAG} determined by $f$, namely $\left| f\right|\colon \left| X\right|\to\left|\mathrm{Sp\acute et}(E) \right| 
		 =\mathrm{Spec}(\pi_0(E))$ factors through the open subscheme 
		 $\mathrm{Spec}(\pi_0(E))\setminus V(I_{t+1})$  (i.e., the locus of height at most $t$).

		 \item \label{it:hc-ker} For every finite $p$-local spectrum $F$ of type greater than $t$, we have $\mathcal{O}_{\mathfrak{X}}\otimes F\simeq 0$.
		 \end{enumerate}
		 In particular, the chromatic height of the global sections $\Gamma({\mathfrak X},{\mathcal O}_{{\mathfrak X}})$ is given by 
	  \begin{equation}\label{height-formula}  \mathrm{ht}(\Gamma({\mathfrak X},{\mathcal O}_{{\mathfrak X}}))=\max \{\mathrm{ht}(({\mathcal G}_\Omega)^{\mathrm{for}})\, \mid \, \Omega\to X\mbox{ a geometric point}\},
	 \end{equation}  
	 	 where $\mathrm{ht}(({\mathcal G}_\Omega)^{\mathrm{for}})$ denotes the height of the formal part of 
	 the base-change of ${\mathcal G}$ to $\Omega$ (along the composition $\Omega\to X\to\mathrm{Spec}(\pi_0(E)$).

 \end{lemma}
 
 \begin{proof} 
	 We show the equivalence between \ref{it:hc-locsheaf} and \ref{it:hc-ker} first: If $R$ is an $L_t$-local ring spectrum and $F$ is a finite $p$-local spectrum of type $\ell>t$, then $R\otimes F\simeq L_t(R)\otimes F\simeq R\otimes L_t(F)\simeq 0$ since $L_t$ is smashing. Applying this with $R={\mathcal O}_{\mathfrak{X}}(U)$ shows that i) implies iii). Conversely, \ref{it:hc-ker} implies that for every $\ell> t$, we have $R\otimes T(\ell)\simeq 0$,
	 and thus $0=K(\ell)^*(R\otimes T(\ell))\simeq K(\ell)^*(R)\otimes_{K(\ell)^*}K(\ell)^*(T(\ell))$. Since $K(\ell)^*(T(\ell))\neq 0$ and $K(\ell)^*$ is a (graded) field, we deduce
	 that $K(\ell)^*(R)=0$. Applying this to the relevant chromatic fracture squares, we see that $L_k(R)\stackrel{\simeq}{\to} L_t(R)$ for all $k\geq t$.  Since $\mathcal{O}_{\mathfrak{X}}$ is a sheaf of $L_n$-local ring spectra, this implies
	  $R\stackrel{\simeq}{\to}L_n(R)\stackrel{\simeq}{\to}L_t(R)$, and hence that \ref{it:hc-locsheaf} holds.
	  	 
	 To establish the equivalence between the first two conditions, we first observe that both \ref{it:hc-locsheaf} and \ref{it:hc-cutout} are \'etale local on $X$:
	 For \ref{it:hc-locsheaf}, this is just the sheaf condition for ${\mathcal O}_{\mathfrak{X}}$ together with the fact that any limit
	 of $L_t$-local spectra is $L_t$-local. For \ref{it:hc-cutout}, this follows more directly because the underlying map of any \'etale cover is surjective.
	 
	 We can thus assume that ${\mathfrak X}=\mathrm{Sp\acute et}(R)$ for some $\mathbb{E}_\infty$-$E$-algebra $E\to R\not \simeq 0$. To settle this special case, we will freely use
	 the notation and results of \cite{greenleesmayMUmodules}. Specifically, we have $L_t(R)\simeq R[I_{t+1}^{-1}]\simeq E[I_{t+1}^{-1}]\otimes_E R$, and there is a fiber sequence of $E$-modules
	 \[ K(I_{t+1})\longrightarrow E\longrightarrow E[I_{t+1}^{-1}]. \]
	 
	 This shows that \ref{it:hc-locsheaf} is equivalent to $K(I_{t+1})\otimes_E R\simeq 0.$
	 Direct inspection of the construction of $K(I_{t+1})$ shows that $K(I_{t+1})\otimes_E R\simeq K(I_{t+1}\cdot\pi_* R)$, and that $K(I_{t+1}\cdot\pi_* R)\simeq 0$ is equivalent
	 to $I_{t+1}\cdot\pi_* R=\pi_*R$, which is equivalent to \ref{it:hc-cutout}.

	 Finally, the formula for $\mathrm{ht}(\Gamma({\mathfrak X},{\mathcal O}_{{\mathfrak X}}))$ in \Cref{height-formula} follows because, for every $0\leq t\leq n$, condition \ref{it:hc-cutout} can be checked on geometric points.
 \end{proof}
 
 \begin{rem} In the above proof we used the well-known implication
 $T(\ell)_*(X)=0\Rightarrow K(\ell)_*(X)=0$, valid for any spectrum $X$.
 The reverse implication is the telescope conjecture, now believed by many to be false.
 One can, however, establish the reverse implication for up-to-homotopy ring spectra,
 using the nilpotence theorem of \cite{Hopkins-Smith}.
 \end{rem}

 Recall that we fixed a finite abelian $p$-group $A$, a family ${\mathcal F}$ of
 its subgroups and a Lubin--Tate spectrum $E$ of height $n$ at $p$.
 To apply \Cref{nlem:heightcheck} to determine the chromatic height of the $\mathbb{E}_\infty$-$E$-algebra $\Phi^{{\mathcal F}}(\underline{E})$,
 we review the modular interpretation of the $\mathbb{E}_\infty$-$E$-algebra $E^{BA_+}$. We will show, 
 in particular, that all $\Phi^{{\mathcal F}}(\underline{E})$ occur as suitable local sections
 of its structure sheaf. So we study in some detail the affine spectral scheme $\mathrm{Sp\acute et}
 (E^{BA_+})$ corresponding
 to the $\mathbb{E}_\infty$-ring $E^{BA_+}$.
 The first step is to recall the determination of the (classical) commutative ring $\pi_0(E^{BA_+}$) in algebro-geometric terms. The commutative ring $\pi_0(E)$ carries a one-dimensional formal group $F$ which is a universal deformation of its special fiber. The system ${\mathcal G}:=(F[p^k])_
{k\ge 0}$ of $p$-power torsion constitutes a $p$-divisible group over $\pi_0(E)$\footnote{More is true: This is part of an equivalence between $p$-divisible commutative formal Lie groups over $\pi_0(E)$ and connected $p$-divisible groups over $\pi_0(E)$ \cite[Prop.~1]{tate}.}.
We denote by $A^*$ the Pontryagin dual of $A$. Choosing $N$ so large that $p^NA=0$, we consider the functor $\underline{\mathrm{Hom}}(A^*,{\mathcal G}[p^N])$ on $\pi_0(E)$-algebras valued in abelian groups, which sends every $R$ to the group of homomorphisms from 
$A^*$ to ${\mathcal G}[p^N](R)$. 

\begin{prop}\label{prop:piEBA}
There is an isomorphism of finite flat group schemes of rank $|A|^n$ over $\pi_0(E)$
\[ \mathrm{Spec}(\pi_0(E^{BA_+}))\simeq \underline{\mathrm{Hom}}(A^*,{\mathcal G}[p^N]).\]
\end{prop}

\begin{proof} The isomorphism is \cite[Prop 5.12]{hkr}. The computation of the rank is immediate, cf. \cite[Sec.7]{strickland-formal}.
\end{proof}

\begin{rem}\
\begin{enumerate}[label = \roman*)]
\item The above can be rephrased by saying that $\mathrm{Sp\acute et}(E^{BA_+})$ is an even periodic enhancement of the composition
  \[ \underline{\mathrm{Hom}}(A^*,{\mathcal G}[p^N])\longrightarrow \mathrm{Spec}(\pi_0(E))
 \longrightarrow M_{FG}\]
 in the sense of \cite[Def.~2.5]{akhil_affiness}.
\item For cyclic $A$, \Cref{prop:piEBA} admits an interesting generalization from the case of $BA=K(A,1)$ to considering $K(A,m)$ for arbitrary $m\ge 1$ instead, see \cite[Thm.~3.4.1]{ambidex}.
\end{enumerate}
\end{rem}

  We next consider the principal open subschemes of $\underline{\mathrm{Hom}}(A^*,{\mathcal G}[p^N])$
 determined by Euler classes.
 We fix a coordinate $t\colon BS^1_+ \longrightarrow E$ for the formal group $F={\mathcal G}^{\mathrm{for}}$, and for every character 
 $\rho\colon A\longrightarrow S^1$ refer to the composition
 \[ e(\rho):=\left( BA_+\xrightarrow{B\rho_+}BS^1_+\xrightarrow{t} E \right)\in\pi_0(E^{BA_+})\]
 as {\em the Euler class of $\rho$}. The principal open subscheme determined by $\rho$ is
 
 \[ U(e(\rho)):=\mathrm{Spec}(\pi_0(E^{BA_+})[e(\rho)^{-1}])\subseteq\mathrm{Spec}(\pi_0(E^{BA_+}))\simeq\underline{\mathrm{Hom}}(A^*,{\mathcal G}[p^N]).\]
 
Observe that there is a closed immersion \[\underline{\mathrm{Hom}}(\mathrm{ker}(\rho)^*,{\mathcal G}[p^N])
\hookrightarrow \underline{\mathrm{Hom}}(A^*,{\mathcal G}[p^N])\] determined by pullback along the surjective
restriction of characters $A^*\to \mathrm{ker}(\rho)^*$. The following \Cref{nprop:principalopens} is a basic observation.
After translating between affine schemes and algebra, the proof is just as in \cite[Prop.~3.20]{greenlees_may_equivariant}. We denote by ${\mathcal O}_{\underline{\mathrm{Hom}}(A^*,{\mathcal G}[p^N])}$ the structure sheaf of the affine derived scheme $\mathrm{Sp\acute et}(\pi_0(E^{BA_+}))$, taking the result of \Cref{prop:piEBA} as an identification in the following.

\begin{prop}\label{nprop:principalopens}
	In the above situation, we have
	\begin{enumerate}[label = \roman*)]
		\item \label{it:principal-cutout} an equality \[U(e(\rho))=\underline{\mathrm{Hom}}(A^*,{\mathcal G}[p^N])\setminus\underline{\mathrm{Hom}}(\mathrm{ker}(\rho)^*,{\mathcal G}[p^N])\] of open subschemes of
		$\underline{\mathrm{Hom}}(A^*,{\mathcal G}[p^N])$ and 

		\item \label{it:principal-sections} an identification of $E^{BA_+}$-algebras \[ \Gamma(U(e(\rho)),{\mathcal O}_{\underline{\mathrm{Hom}}(A^*,{\mathcal G}[p^N])})\simeq \Phi^{[\leq\mathrm{ker}(\rho)]}(\underline{E}),\]
		where $[\leq\mathrm{ker}(\rho)]:=\{ A'\subseteq A\,\mid\, A'\subseteq\mathrm{ker}(\rho)\}$ denotes the family of
		subgroups of $A$ on which $\rho$ vanishes.
	\end{enumerate}
\end{prop}

\begin{proof}  
	To prove \ref{it:principal-cutout}, we observe there is an obvious cartesian square
	\[ \xymatrix{ \underline{\mathrm{Hom}}(\mathrm{ker}(\rho)^*,{\mathcal G}[p^N]) \ar[d]
	\ar[r] \ar@{^{(}->}[r] & \underline{\mathrm{Hom}}(A^*,{\mathcal G}[p^N]) \ar[d] \\
	\{ 0 \} \ar@{^{(}->}[r] & \underline{\mathrm{Hom}}(\mathrm{im}(\rho)^*,{\mathcal G}[p^N]) .}\]
	The zero section of $\underline{\mathrm{Hom}}(\mathrm{im}(\rho)^*,{\mathcal G}[p^N])$
	is given by the vanishing of the Euler class $e(\mathrm{im}(\rho)\subseteq S^1)$ (because $\mathrm{im}(\rho)$ is cyclic of $p$-power order),
	and the naturally of $e(-)$ implies that the closed immersion 
	$\underline{\mathrm{Hom}}(\mathrm{ker}(\rho)^*,{\mathcal G}[p^N])\hookrightarrow\underline{\mathrm{Hom}}(A^*,{\mathcal G}[p^N])$ is given by the vanishing of $e(\rho)$.

	For the proof of \ref{it:principal-sections}, we will need a bit more information about $e(\rho)$. We will abuse notation and let $\rho$ denote the corresponding complex representation. Applying one point compactification to the inclusion $0\hookrightarrow \rho$, we obtain a map of based $A$-spaces $e(\rho)^\prime \colon S^0\to S^{\rho}$. Smashing $e(\rho)^\prime$ with $\underline{E}$, taking $A$ fixed points, and using our complex orientation, we obtain the $E^{BA_+}$-module map $E^{BA_+}\to E^{BA_+}$ corresponding to the map $e(\rho)$ above \cite[\S 5.1]{MNN}. 

	Now we do have $\Gamma(U(e(\rho)),{\mathcal O}_{\underline{\mathrm{Hom}}(A^*,{\mathcal G}[p^N])})\simeq
	E^{BA_+}[e(\rho)^{-1}]$ by the construction of the structure sheaf ${\mathcal O}_{\underline{\mathrm{Hom}}(A^*,{\mathcal G}[p^N])}$. To see the claim, we will use the equivalence $\wt{E}[\leq\mathrm{ker}(\rho)]\simeq\mathrm{colim}_n S^{n\rho}$ (which can be checked using \Cref{ndef:prank}, \ref{it:families})
	with transition maps given by multiplication with $e(\rho)^\prime$. Using this, we see
	\begin{align*}
		 \Phi^{[\leq\mathrm{ker}(\rho)]}(\underline{E}) &= \left( \wt{E}[\leq\mathrm{ker}(\rho)]\otimes\underline{E}\right)^A\\
		 	&\simeq  \mathrm{colim}_n \left( \underline{E}^A\xrightarrow{\cdot e(\rho)} \underline{E}^A\xrightarrow{\cdot e(\rho)}\cdots\right)\\
		 	& \simeq \underline{E}^A[e(\rho)^{-1}]\\
		 	& \simeq E^{BA_+}[e(\rho)^{-1}]\\
		 	& \simeq \Gamma(U(e(\rho)),{\mathcal O}_{\underline{\mathrm{Hom}}(A^*,{\mathcal G}[p^N])}).\qedhere
	\end{align*} 
\end{proof}
 
We want to generalize \Cref{nprop:principalopens}, \ref{it:principal-sections} by finding an open subscheme (typically non-principal) of $\underline{\mathrm{Hom}}(A^*,{\mathcal G}[p^N])$ over which 
the sections are $\Phi^{{\mathcal F}}(\underline{E})$ for a given family ${\mathcal F}$. We also want this open subscheme to have a modular interpretation, as in \Cref{nprop:principalopens}, \ref{it:principal-cutout} above, which will ultimately allow for the height computation of \Cref{nthm:heightdrop}.

We begin by observing that for every family ${\mathcal F}$ we have 
\begin{equation}\label{neq:familydecomp}
	{\mathcal F}=\bigcup_{A'\in{\mathcal F}}[\leq A'] = \bigcup _{A'\in{\mathcal F}}\bigcap_{\rho\in (A/A')^*}[\leq\mathrm{ker}(\rho)].
\end{equation}
The first equality in \eqref{neq:familydecomp} is trivial and the second one follows from duality of finite abelian groups.
Here, we commit a mild abuse of notation by identifying some $\rho\in(A/A')^*$ with the composition $A\to A/A' \xrightarrow{\rho} S^1$.

The decomposition (\ref{neq:familydecomp}) suggests to consider the following open 
subscheme of $\underline{\mathrm{Hom}}(A^*,{\mathcal G}[p^N])$:
\begin{equation}\label{neq:UFdecomp}
	U({\mathcal F}):= \bigcap_{A'\in{\mathcal F}}\bigcup_{\rho\in(A/A')^*} U(e(\rho)).
\end{equation}
By \Cref{nprop:principalopens}, \ref{it:principal-cutout}, this equals 
\begin{equation}\label{neq:UFexplicit}
	U({\mathcal F})=\underline{\mathrm{Hom}}(A^*,{\mathcal G}[p^N])\setminus\left(   \bigcup_{A'\in{\mathcal F}}\bigcap_{\rho\in(A/A')^*}  \underline{\mathrm{Hom}}(\mathrm{ker}(\rho)^*,{\mathcal G}[p^N]) \right).
\end{equation}
This last equation gives us the desired modular interpretation of $U({\mathcal F})$, and we can also identify the sections over it, as follows.

\begin{lemma}\label{nlem:section}
	In the above situation, we have
	\[ \Gamma(U({\mathcal F}),{\mathcal O}_{\underline{\mathrm{Hom}}(A^*,{\mathcal G}[p^N])})\simeq \Phi^{{\mathcal F}}(\underline{E})\]
	as algebras over $E^{BA_+}=\Gamma({\underline{\mathrm{Hom}}(A^*,{\mathcal G}[p^N])},{\mathcal O}_{\underline{\mathrm{Hom}}(A^*,{\mathcal G}[p^N])})$.
\end{lemma}

\begin{proof} 
	For every open $U\subseteq X:=\underline{\mathrm{Hom}}(A^*,{\mathcal G}[p^N])$
	and $e\in\pi_0(E^{BA_+})=\pi_0\left(\Gamma(X,{\mathcal O}_X)\right)$ we have a cartesian square

	\[ \xymatrix{ \Gamma(U\cup U(e),{\mathcal O}_X)\ar[r] \ar[d] & \Gamma(U,{\mathcal O}_X) \ar[d] \\
	\Gamma(U(e),{\mathcal O}_X) \ar[r] & \Gamma(U\cap U(e),{\mathcal O}_X)\simeq \Gamma(U,{\mathcal O}_X)[e^{-1}]. } \]

	Given any family $\mathcal F$ and any character $\rho\colon A\longrightarrow S^1$, we also have a cartesian square of genuine $G$-spectra

	\[ \xymatrix{  \Sigma^\infty\left( \wt{E}\left( {\mathcal F}\cap [\leq\mathrm{ker}(\rho)]  \right)\right) \ar[r] \ar[d] & \ar[d] \Sigma^\infty\left( \wt{E}{\mathcal F} \right) \\
	\Sigma^\infty\left( \wt{E}\left( [\leq\mathrm{ker}(\rho)] \right)\right)  \ar[r] & \Sigma^\infty\left( \wt{E}\left( \mathcal F\cup [\leq\mathrm{ker}(\rho)]\right)\right)  \simeq  \Sigma^\infty\left(\wt{E}{\mathcal F}\right)\otimes \Sigma^\infty\left( \wt{E}\left( [\leq\mathrm{ker}(\rho)]\right)\right).} \]

	Assume now that $U$ and $\mathcal \cF$ are such that $\Gamma(U,{\mathcal O}_X)\simeq\Phi^{\mathcal F}(\underline{E})$ as $E^{BA_+}$-algebras. Comparison of the above two cartesian squares (with $e:=e(\rho)$) and using \Cref{nprop:principalopens}, ii) then implies that $\Gamma(U\cup U(e(\rho)),{\mathcal O}_X)\simeq \Phi^{{\mathcal F}\cap[\leq\mathrm{ker}(\rho)]}(\underline{E})$ as $E^{BA_+}$-algebras. Applying this inductively with \Cref{nprop:principalopens}, \ref{it:principal-sections} as a starting point, we see that for every
	$A'\in{\mathcal F}$, we have
	\begin{equation}\label{neq:first}
		\Gamma\left( \bigcup_{\rho\in(A/A')^*} U(e(\rho)), {\mathcal O}_X\right)\simeq
		\Phi^{\bigcap_\rho [\leq\mathrm{ker}(\rho)]}(\underline{E})\stackrel{(\ref{neq:familydecomp})}{\simeq} \Phi^{[\leq A']}(\underline{E}).
	\end{equation}
	Finally, for any two opens $U,V\subseteq X$, we have 
	\begin{equation}\label{neq:intersect}
		\Gamma(U\cap V,{\mathcal O}_X)\simeq \Gamma(U,{\mathcal O}_X)\otimes_{\Gamma(X,{\mathcal O}_X)} \Gamma(V,{\mathcal O}_X),
	\end{equation}
	and conclude
	\begin{align*}
	 \Gamma(U({\mathcal F}),{\mathcal O}_X) & \stackrel{(\ref{neq:UFdecomp})}{\simeq}
	\Gamma\left(\bigcap_{A'\in{\mathcal F}}\bigcup_{\rho\in(A/A')^*} U(e(\rho)),
	{\mathcal O}_X\right) \\ 
	& \stackrel{(\ref{neq:intersect})}{\simeq} \bigotimes_{A'\in{\mathcal F}} \Gamma(\bigcup_{\rho\in(A/A')^*} U(e(\rho)), {\mathcal O}_X) \\
	& \stackrel{(\ref{neq:first})}{\simeq} \bigotimes_{A'\in{\mathcal F}} \Phi^{[\leq A']}(\underline{E}) \\
	 & \stackrel{(*)}{\simeq} \Phi^{\bigcup_{A'\in{\mathcal F}}[\leq A']}(\underline{E})\\
	 & \stackrel{(\ref{neq:familydecomp})}{=}\Phi^{{\mathcal F}}(\underline{E}).
	\end{align*}
	For the equivalence $(*)$ we used $\wt{E}{\mathcal F}_1 \otimes \wt{E}{\mathcal F}_2\simeq
	\wt{E}({\mathcal F}_1 \cup {\mathcal F}_2)$, which is again easily inferred from \Cref{ndef:prank}, \ref{it:families}.
\end{proof}

Now we have assembled everything to prove the main result of this subsection.

\begin{proof}[Proof of \Cref{nthm:heightdrop}]
	Taking the identification of \Cref{nlem:section} and applying \Cref{height-formula} with ${\mathfrak X}:=(U({\mathcal F}),{\mathcal O}_{\underline{\mathrm{Hom}}(A^*,{\mathcal G}[p^N])}\mid_{U({\mathcal F})}) $, we find that
	\[
	\mathrm{ht}(\Phi^{{\mathcal F}}(\underline{E}))=\max \{ \mathrm{ht}(({\mathcal G}_\Omega)^{\mathrm{for}})\,\mid\, \Omega\to U({\mathcal F})\mbox{ a geometric point} \}.
	\]
	If $\Omega$ is any geometric point of $\mathrm{Spec}(\pi_0(E))$, then we have
	${\mathcal G}[p^N](\Omega)\simeq\left(\mathbb{Z}/p^N\mathbb{Z}\right)^{n-\mathrm{ht}(({\mathcal G}_\Omega)^{\mathrm{for}})}$, so using (\ref{neq:UFexplicit}), we obtain the following cumbersome, but elementary, description of the sought for
	$\mathrm{ht}(\Phi^{{\mathcal F}}(\underline{E}))$: It is the largest $t\ge 0$
	such that there is a homomorphism $\varphi\colon A^*\to\left(\mathbb{Z}/p^N\mathbb{Z}\right)^{n-t}$ such that for all $A'\in{\mathcal F}$ there is a character
	$\rho\in(A/A')^*$ such that $\mathrm{ker}(A^*\to\mathrm{ker}(\rho)^*)\not\subseteq
	\mathrm{ker}(\varphi)$.

	Observe that \[ \left\{ \mathrm{ker}(A^*\to\mathrm{ker}(\rho)^*) = \mathrm{im}(\rho)^* \,\mid\, \rho\in(A/A')^*\right\}=
	\left\{ C\subseteq (A/A')^*\subseteq A^*\mbox{ cyclic}\right\}.\]
	So the condition on $\varphi$ is that it does not vanish on $(A/A')^*$, for every $A'\in {\mathcal F}$. We now determine which $C\subseteq A^*$ can occur as the kernels of such $\varphi$. Since $p^NA=0$, for any subgroup $C\subseteq A^*$ there will be an inclusion $A^*/C\hookrightarrow \left(\mathbb{Z}/p^N\mathbb{Z}\right)^{n-t}$ if and only if $\prank(A^*/C)\leq n-t$. 

	Now a maximal $t$ satisfying the above is determined by:
	 \[ n-t=\min\left\{ \prank(A^*/C)\,\mid\, C\subseteq A^*\mbox{ s.t. }\forall A'\in{\mathcal F}, (A/A')^*\not\subseteq C \right\}. \]
	 Using that for any finite abelian $p$-group $B$ we have $\prank(B)=  \prank(B^*)$,
	 that $A'\mapsto (A/A')^*$ and $C\mapsto (A^*/C)^*\subseteq A^{**}=A$ are mutually inverse inclusion reversing bijections between subgroups of $A$ and $A^*$, and that 
	 ${\mathcal F}$ is a family, this becomes
	  \[ n-t=\min\left\{ \prank(\wt{C})\,\mid\, \wt{C}\subseteq A, \wt{C}\not\in{\mathcal F}\right\}=\coprank({\mathcal F}), \]
	 as claimed.
 \end{proof}
 
We finally use \Cref{nthm:heightdrop} to prove \Cref{cor:bound-height-drop}, which shows that the formation
 of geometric fixed points can lower the type of a finite complex at most by the $p$-rank of the group acting. We repeat the statement for convenience 
 
\begin{cor}\label{for:bound-height-drop-restate}
	Assume $A$ is a finite abelian $p$-group and $X\in\sp_{A,(p)}^\omega$. Then type$(\Phi^A(X))\ge\mathrm{type}(\Phi^{\{0\}}(X))-\prank(A)$. 
\end{cor}

\begin{proof}
	Let $k:=\prank(A)$, and we can assume that $n:= \mathrm{type}(\Phi^{\{0\}}(X))\ge k+1$ for
	otherwise our assertion is vacuously true.
	Our assumption is that $K(n-1)_*(\Phi^{\{ 0 \}}(X))=0$; in the case $n=\infty$, i.e., $\Phi^{\{ 0 \}}(X)$ is contractible, the following argument is applied for every Morava $K$-theory. 

	Denote by $E$ a Lubin--Tate spectrum at $p$ of height $n-1$.
	Since $\mathrm{type}(\Phi^{\{0\}}(X)) > n-1$ and any map into a $K(n-1)$-local spectrum factors through the $K(n-1)$-localization, we have $E^*(\Phi^{\{0\}}(X))=0$.
	Then we have more generally $E^*(EA'_+\otimes_{A'} X)=0$ for all subgroups $A'\subseteq A$ by considering the collapsing  homotopy fixed point spectral sequence. In other words, the Borel  spectrum $F(X, \underline{E})$ is equivariantly contractible because we have
	$\pi_*\left(F(X, \underline{E})^{A'}\right)=E^{-*}(EA'_+\otimes_{A'} X)=0$. Equivalently, all of the geometric fixed points of this spectrum are contractible. 
	We conclude that \[ 0\simeq \Phi^A(F(X, \underline{E}))\simeq D(\Phi^A(X))\otimes \Phi^A(\underline{E}).\] 
	The second of these equivalences uses the finiteness of $X$ and the fact that $\Phi^A$ is a symmetric monoidal functor. We know that the chromatic
	height of $\Phi^A(\underline{E})$ is $n-k-1$ by \Cref{rem:height-drop}, \ref{it:geom-ht-drop}, so its $p$-local finite acyclics are precisely the 
	complexes of type at least $n-k$, cf.~\Cref{nlem:heightcheck}, \ref{it:hc-ker}. This means that type$(\Phi^A(X))=$type$(D(\Phi^A(X)))\ge n-k$ (also in case $n=\infty$)\footnote{Poincar\'e duality for $K(n)$-cohomology makes it clear that the types of a finite complex and its dual agree.}.\end{proof}

\section{The complexes of Arone and Lesh} \label{subsec:arone-lesh}
 
 The aim of this subsection is to use work of Arone, Dwyer, Lesh, and Mahowald to give a proof of \Cref{prop:complexes}, the statement of which we repeat for convenience:
 
\begin{thm}\label{prop:complexes-restate}(Arone, Dwyer, Lesh, Mahowald)\\
Let $p$ be a prime, $n\ge 1$ and $\Delta:=(\mathbb{Z}/p\mathbb{Z})^{\times n}$ the corresponding elementary abelian $p$-group. Then, there is a $p$-local finite $\Delta$-equivariant spectrum $F(n)\in\sp_{\Delta,(p)}^\omega$ satisfying the following conditions:
	\begin{enumerate}[label = \roman*)]
		\item The geometric fixed points $\Phi^{\Delta}(F(n))$ have type $0$.
		\item The underlying non-equivariant spectrum of $F(n)$, i.e., $\Phi^{\left\{ 0\right\}}(F(n))$, has type $n$.
	\end{enumerate}
\end{thm}

 This result is of central importance to the entire paper.
 We remark that the complexes $F(n)$ are closely related
 to the complexes constructed by Mitchell in \cite{Mit85}, which were the first examples of finite (non-equivariant) complexes of arbitrary type. The exact relation between these
 two families of complexes is worked out in \cite[Sec.~2]{Aro98}.

The construction of $F(n)$ uses equivariant homotopy theory for compact Lie groups, and we refer the reader
to \cite{MNN} for a rapid review of this. We first describe the groups involved.
Fix an integer $m\ge 1$ and let $U(m)$ denote the unitary group of rank $m$.
Embed the permutation group $\Sigma_m\subseteq U(m)$ as the subgroup of permutation matrices. Fix the (non-standard) embedding $U(m-1)\subseteq U(m)$
corresponding to the embedding $\mathbb{C}^{m-1}\subseteq \mathbb{C}^m$ as the orthogonal complement of the diagonal. Note that then $\Sigma_m\subseteq U(m-1)
\subseteq U(m)$. 

Now consider the case $m=p^n$ and embed
$\Delta\subseteq\Sigma_{p^n}$ using the regular representation of $\Delta$ on its
 underlying set.
Let ${\mathcal P}_{p^n}$ denote the (geometric realization of the) poset of
non-trivial proper partitions of a set with $p^n$ elements. The equivariant Spanier--Whitehead dual $D({\mathcal P}_{p^n}^\diamond)$ of its unreduced suspension is
 canonically an object of $\sp_{\Sigma_{p^n}}^\omega$, i.e., a finite genuine 
 $\Sigma_{p^n}$-spectrum, and we consider its twist
 $D({\mathcal P}_{p^n}^\diamond)\otimes S^{\overline{\rho}}$ by the representation
 sphere of the reduced regular representation ${\overline{\rho}}$ of $\Sigma_{p^n}$.
 We denote by $\mathrm{Ind}_{\Sigma_{p^n}}^{U(p^n-1)}\colon \sp_{\Sigma_{p^n}} \to \sp_{U(p^n-1)}$ the induction functor and by $\mathrm{Res}_{\Delta}^{U(p^n-1)}\colon \sp_{U(p^n-1)} \to \sp_{\Delta}$ the restriction functor. Note that both induction and
 restriction preserve finite spectra.
 We can thus finally define \[ F(n):=\left( \mathrm{Res}_{\Delta}^{U(p^n-1)} \left( \mathrm{Ind}_{\Sigma_{p^n}}^{U(p^n-1)} (D({\mathcal P}_{p^n}^\diamond)\otimes S^{\overline{\rho}})\right)\right)_{(p)}\in\sp_{\Delta,(p)}^\omega.\]
 
The geometric fixed points of $F(n)$ will be analyzed through the following result, which is essentially contained in \cite[\S 5]{ArL17}.
 
 \begin{prop}\label{prop:geometricinduct}
 We have \[   \Sigma^\infty\left( U(1)^{\times (p^n-1)}_+\right) \otimes \left(\Phi^{\Delta}(-)\right)\stackrel{\simeq}{\longrightarrow} \Phi^{\Delta}\circ\mathrm{Ind}_{\Sigma_{p^n}}^{U(p^n-1)}(-)\]
 as functors $\sp_{\Sigma_{p^n}}\longrightarrow\sp$.  \end{prop}
 
 \begin{proof}
We will first establish the following unstable refinement of our assertion.
Denote by $C:=C_{U(p^n-1)}(\Delta)\simeq U(1)^{p^n-1}$ the centralizer of $\Delta$ in $U(p^n-1)$, and note
that $\Delta\subseteq C$. For every $\Sigma_{p^n}$-space $X$, we define a map 
between fixed point spaces
\begin{equation}\label{eqn:al}
\varphi_X\colon C/\Delta \times X^\Delta\longrightarrow\left( U(p^n-1)\times_{\Sigma_{p^n}} X\right)^\Delta\, , \, \varphi_X( c\Delta,x):=[c,x], 
\end{equation}
which we claim is an equivalence, functorial in $X$.

It is immediate that \eqref{eqn:al} is well-defined and functorial in $X$. Its injectivity results from 
an easy computation, using that $\Delta\subseteq \Sigma_{p^n}$ is its own centralizer: If $[c,x]=[c',x']$, then there exists some $\sigma \in \Sigma_{p^n}$ such that \[(c,x)=(c'\sigma^{-1}, \sigma x')\in U(p^n-1)\times X.\] This implies $c\sigma=c'\in C$ and hence $c\sigma d \sigma^{-1}  c^{-1} = d$ for all $d\in \Delta$. From this we see that $\sigma\in \Sigma_{p^n}\cap C=\Delta$ and hence $c = c' \sigma^{-1} $ for some $\sigma^{-1}\in \Delta$ and $x=\sigma x'=x'$, as desired.

To see the surjectivity of \eqref{eqn:al}, recall the general computation of the fixed points of an induced space
\[ \left( U(p^n-1)\times_{\Sigma_{p^n}} X\right)^\Delta = \left\{ [u,x] \,\mid\, u\in U(p^n-1)
\mbox{ such that } u^{-1}\Delta u\subseteq \Sigma_{p^n}\mbox{ and }x\in X^{u^{-1}\Delta u} \right\}.\]
Denoting by $N:=N_{U(p^n-1)}(\Delta; \Sigma_{p^n}):=\left\{ u\in U(p^n-1)\,\mid\, u^{-1}\Delta u\subseteq \Sigma_{p^n}\right\}$, Arone and Lesh show in \cite[\S 5]{ArL17}, that the obvious inclusion $C\cdot\Sigma_{p^n}\subseteq N$ is in fact an equality. Now, given $[u,x]\in \left( U(p^n-1)\times_{\Sigma_{p^n}} X\right)^\Delta$, we can write $u=c\sigma$ with $c\in C$, $\sigma\in\Sigma_{p^n}$ and $x\in X^{u^{-1}\Delta u}=X^{\sigma^{-1}\Delta \sigma}$. We conclude that $x=\sigma^{-1}y$ for some $y\in X^\Delta$, and then compute
\[ \varphi_X(c\Delta,y)=[c,y]=[c\sigma,\sigma^{-1}y]=[u,x],\]
hence $\varphi_X$ is indeed surjective.

To identify $C/\Delta$, we observe that $\Delta\to C\simeq U(1)^{\times(p^n-1)}$ maps to each of the $p^n-1$ components by each of the non-trivial irreducible representations of $\Delta$. The quotient of each of these actions is \[U(1)/\Delta = U(1)/\mathrm{Im}(\Delta) = U(1)/(\mathbb{Z}/p\mathbb{Z})\simeq U(1).\] We can thus identify the quotient $U(1)^{\times (p^n-1)}/\Delta^{\times (p^n-1)}\simeq U(1)^{\times (p^n-1)}$. To identify the  quotient $C/\Delta\simeq U(1)^{\times (p^n-1)}/\Delta$ by the diagonal copy of $\Delta$ we consider the induced quotient fiber sequence:
\[ \Delta^{\times (p^n-1)}/\Delta \to U(1)^{\times(p^n-1)}/\Delta \to U(1)^{\times (p^n-1)}/\Delta^{\times (p^n-1)}(\simeq U(1)^{\times (p^n-1)}). \] 
This forces $U(1)^{\times(p^n-1)}/\Delta$ to be a $K(\mathbb{Z}^{\times (p^n-1)},1)\simeq U(1)^{\times (p^n-1)}$. 

In order to pass to the stable setting, we observe that the above equivalence (\ref{eqn:al}) has an obvious analogue for pointed $\Sigma_{p^n}$-spaces which implies a natural stable equivalence for pointed $\Sigma_{p^n}$-spaces $X$ of the form
\begin{equation}\label{eq:al2}
 \Sigma^\infty\left( C/\Delta_+\right) \otimes \Sigma^\infty(X^\Delta)\xrightarrow{\simeq}\Sigma^\infty\left( \left( U(p^n-1)_+\wedge_{\Sigma_{p^n}} X\right)^\Delta\right) \simeq \Phi^{\Delta} \left(\Sigma^\infty (U(p^n-1)_+)\otimes_{\Sigma_{p^n}} \Sigma^\infty (X)\right).
\end{equation}
Now by the equivariant Freudenthal suspension theorem \cite[\S IX]{May96} and a standard induction argument over cells, every finite $\Sigma_{p^n}$-spectrum $Y$ is equivalent to $S^{-V}\otimes \Sigma^\infty(X)$ where $\Sigma^\infty(X)$ is the equivariant suspension spectrum of a finite pointed $\Sigma_{p^n}$-space $X$ and $S^{-V}$ is the equivariant Spanier--Whitehead dual of a representation sphere of $\Sigma_{p^n}$. Applying the equivalence of \eqref{eq:al2} to such an $X$ we see that:
\[ \Sigma^\infty(C/\Delta_+)\otimes\Phi^\Delta(Y)\simeq S^{-V^{\Delta}}\otimes \Sigma^\infty (C/\Delta_+) \otimes \Sigma^\infty(X^\Delta)\]
\[\xrightarrow{\simeq} S^{-V^{\Delta}}\otimes \Phi^{\Delta} \left(\Sigma^\infty (U(p^n-1)_+)\otimes_{\Sigma_{p^n}} \Sigma^\infty(X)\right)\simeq \Phi^\Delta(\Ind_{\Sigma_{p^n}}^{U(p^n-1)}(Y)).\]
Finally, every $\Sigma_{p^n}$-spectrum is a filtered colimit of finite $\Sigma_{p^n}$-spectra and hence we obtain the desired equivalence for any $\Sigma_{p^n}$-spectrum, taking into account that $C/\Delta\simeq U(1)^{\times(p^n-1)}$.
 \end{proof}

\begin{proof}[Proof of \Cref{prop:complexes}]
Recall that we defined  \[ F(n):=\left( \mathrm{Res}_{\Delta}^{U(p^n-1)}\left( \mathrm{Ind}_{\Sigma_{p^n}}^{U(p^n-1)} (D({\mathcal P}_{p^n}^\diamond)\otimes S^{\overline{\rho}})\right)\right)_{(p)}\in\sp_{\Delta,(p)}^\omega,\]
	and now need to check properties \ref{it:adlm-0} and \ref{it:adlm-n} for $F(n)$. 

	By \Cref{prop:geometricinduct}, we see that
	\[ \Phi^\Delta(F(n))\simeq\ U(1)_+^{\times(p^n-1)}\otimes\Phi^\Delta (D({\mathcal P}_{p^n}^\diamond)\otimes S^{\overline{\rho}})_{(p)}. \]
	Observe that $\Phi^\Delta(-)$ is symmetric monoidal, and in particular commutes
	over Spanier--Whitehead duals.
Since ${\mathcal P}_{p^n}^\Delta$ is the Tits-building of Gl$_n(\mathbb{F}_p)$
(cf.~\cite[Lem.~10.1]{ADL16}) which is a wedge of $p^{\frac{n(n-1)}{2}}$ spheres of dimension $n-2$ if $n\ge 2$ (for $n=1$, the Tits-building is empty and its unreduced suspension is $S^0$), and $\Phi^\Delta(S^{\overline{\rho}})\simeq S^0$, we see that $\Phi^\Delta(F(n))$ has type zero.

Showing property \ref{it:adlm-n} requires much harder previous work of Arone and Mahowald. Using that $\Phi^{\{0\}}\circ\mathrm{Ind}_{\Sigma_{p^n}}^{U(p^n-1)}\simeq
U(p^n-1)_+\otimes_{h\Sigma_{p^n}}	(-)$, we see that the underlying spectrum of $F(n)$
is given as
\[ \Phi^{\{0\}}(F(n))\simeq U(p^n-1)_+\otimes_{h\Sigma_{p^n}}(D({\mathcal P}_{p^n}^\diamond)\otimes S^{\overline{\rho}})_{(p)}.\]
This is easily recognized to be, up to a shift, one of the spectra figuring in \cite{Aro98} (where Arone uses $K_n$ to denote the suspension of $\cP_n^{\diamond}$).  In \cite[Thm.~0.4]{Aro98}, building on work of Arone and Mahowald~\cite{ArM99}, it is shown that $H^*(\Phi^{\{ 0 \}}(F(n)),\mathbb{F}_p)$ is finitely generated, free and non-zero over the subalgebra $\mathcal{A}_{n-1}$ of the mod $p$ Steenrod algebra.

Using this to compute connective Morava $K$-theories through the Adams spectral sequence,
it easily follows that the type of $\Phi^{\{ 0 \}}(F(n))$ is at least $n$ (see the proof of \cite[Thm.~4.8]{Mit85} for details of this
argument). Since we have already seen that the $\Delta$-geometric fixed points of $F(n)$ have type zero, we know that the type of $\Phi^{\{ 0 \}}(F(n))$ can be at most $n$ by \Cref{cor:bound-height-drop}.
So it must be exactly $n$, which concludes the proof. 
\end{proof}

\begin{rem}\label{rem:akhil}
The case $n=1$ of \Cref{prop:complexes} admits a nice direct proof, which we learned from Akhil Mathew: Let $F(1)'\in\sp_{\mathbb{Z}/p\mathbb{Z}}^\omega$ denote the cofiber of the transfer
	map $t\colon S^0\to \mathbb{Z}/p\mathbb{Z}_+\to  S^0$. On geometric fixed points, this map is zero because it factors through $\Phi^{\mathbb{Z}/p\mathbb{Z}}(\mathbb{Z}/p\mathbb{Z}_+)\simeq 0$, hence $\Phi^{\mathbb{Z}/p\mathbb{Z}}(F(1)')\simeq S^0\vee S^1$ has type zero. The non-equivariant map underlying $t$ is simply multiplication by $p$, hence $\Phi^{\{ 0 \}}(F(1)')$ is equivalent to $S^0/p$, and thus has type one.
\end{rem}

\bibliographystyle{alpha}
\bibliography{balmer-sanders-niko}

\def\cprime{$'$}
\begin{thebibliography}{GHMR05}

\bibitem[ADL16]{ADL16}
Gregory Arone, William~G. Dwyer, and Kathryn Lesh.
\newblock Bredon homology of partition complexes.
\newblock {\em Doc. Math.}, 21:1227--1268, 2016.

\bibitem[AL17]{ArL17}
Gregory Arone and Kathryn Lesh.
\newblock Fixed points of coisotropic subgroups of {$\Gamma_{k}$} on
  decomposition spaces.
\newblock 2017.
\newblock Available at \url{http://front.math.ucdavis.edu/1701.06070}.

\bibitem[AM99]{ArM99}
Gregory Arone and Mark Mahowald.
\newblock The {G}oodwillie tower of the identity functor and the unstable
  periodic homotopy of spheres.
\newblock {\em Invent. Math.}, 135(3):743--788, 1999.

\bibitem[AMS98]{ando-morava-sadofsky}
Matthew Ando, Jack Morava, and Hal Sadofsky.
\newblock Completions of {$\mathbf Z/(p)$}-{T}ate cohomology of periodic
  spectra.
\newblock {\em Geom. Topol.}, 2:145--174, 1998.

\bibitem[Aro98]{Aro98}
Gregory Arone.
\newblock Iterates of the suspension map and {M}itchell's finite spectra with
  {$A_k$}-free cohomology.
\newblock {\em Math. Res. Lett.}, 5(4):485--496, 1998.

\bibitem[Bal02]{balmer_reconstruction}
Paul Balmer.
\newblock Presheaves of triangulated categories and reconstruction of schemes.
\newblock {\em Math. Ann.}, 324(3):557--580, 2002.

\bibitem[Bal05]{balmer-the-spectrum}
Paul Balmer.
\newblock The spectrum of prime ideals in tensor triangulated categories.
\newblock {\em J. Reine Angew. Math.}, 588:149--168, 2005.

\bibitem[Bal10]{spectra3}
Paul Balmer.
\newblock Spectra, spectra, spectra---tensor triangular spectra versus
  {Z}ariski spectra of endomorphism rings.
\newblock {\em Algebr. Geom. Topol.}, 10(3):1521--1563, 2010.

\bibitem[BCR97]{bcr_modular}
D.~J. Benson, Jon~F. Carlson, and Jeremy Rickard.
\newblock Thick subcategories of the stable module category.
\newblock {\em Fund. Math.}, 153(1):59--80, 1997.

\bibitem[BIK11]{bik_modular}
David~J. Benson, Srikanth~B. Iyengar, and Henning Krause.
\newblock Stratifying modular representations of finite groups.
\newblock {\em Ann. of Math. (2)}, 174(3):1643--1684, 2011.

\bibitem[BS17]{Balmer-Sanders}
Paul Balmer and Beren Sanders.
\newblock The spectrum of the equivariant stable homotopy category of a finite
  group.
\newblock {\em Invent. Math.}, 208(1):283--326, 2017.

\bibitem[DFHH14]{behrens-tmf}
Christopher~L. Douglas, John Francis, Andr\'e~G. Henriques, and Michael~A.
  Hill, editors.
\newblock {\em Topological modular forms}, volume 201 of {\em Mathematical
  Surveys and Monographs}.
\newblock American Mathematical Society, Providence, RI, 2014.

\bibitem[GHMR05]{goerss-res}
Paul Goerss, Hans-Werner Henn, Mark Mahowald, and Charles Rezk.
\newblock A resolution of the {$K(2)$}-local sphere at the prime 3.
\newblock {\em Ann. of Math. (2)}, 162(2):777--822, 2005.

\bibitem[GM95a]{greenleesmayMUmodules}
John P.~C. Greenlees and J.~Peter May.
\newblock Completions in algebra and topology.
\newblock In {\em Handbook of algebraic topology}, pages 255--276.
  North-Holland, Amsterdam, 1995.

\bibitem[GM95b]{greenlees_may_equivariant}
John P.~C. Greenlees and J.~Peter May.
\newblock Equivariant stable homotopy theory.
\newblock In {\em Handbook of algebraic topology}, pages 277--323.
  North-Holland, Amsterdam, 1995.

\bibitem[GM95c]{greenlees-may-tate}
John P.~C. Greenlees and J.~Peter May.
\newblock Generalized {T}ate cohomology.
\newblock {\em Mem. Amer. Math. Soc.}, 113(543):viii+178, 1995.

\bibitem[GS96]{Greenlees-Sadofsky}
John P.~C. Greenlees and Hal Sadofsky.
\newblock The {T}ate spectrum of {$v_n$}-periodic complex oriented theories.
\newblock {\em Math. Z.}, 222(3):391--405, 1996.

\bibitem[Hah16]{hahnheight}
Jeremy Hahn.
\newblock On the {B}ousfield classes of {$H_\infty$}-ring spectra.
\newblock 2016.
\newblock Available at \url{http://front.math.ucdavis.edu/1612.04386}.

\bibitem[HHR16]{hhr_kervaire}
M.~A. Hill, M.~J. Hopkins, and D.~C. Ravenel.
\newblock On the nonexistence of elements of {K}ervaire invariant one.
\newblock {\em Ann. of Math. (2)}, 184(1):1--262, 2016.

\bibitem[HKR00]{hkr}
Michael~J. Hopkins, Nicholas~J. Kuhn, and Douglas~C. Ravenel.
\newblock Generalized group characters and complex oriented cohomology
  theories.
\newblock {\em J. Amer. Math. Soc.}, 13(3):553--594, 2000.

\bibitem[HL13]{ambidex}
Michael~J. Hopkins and Jacob Lurie.
\newblock Ambidexterity in {$K(n)$}-local stable homotopy theory.
\newblock December 2013.
\newblock Available at
  \url{http://www.math.harvard.edu/~lurie/papers/Ambidexterity.pdf}.

\bibitem[Hop87]{Hop87}
Michael~J. Hopkins.
\newblock Global methods in homotopy theory.
\newblock In {\em Homotopy Theory-Proc. Durham Symposium 1985}. Cambridge
  University Pres, 1987.

\bibitem[HS96]{Hovey-Sadofsky}
Mark Hovey and Hal Sadofsky.
\newblock Tate cohomology lowers chromatic {B}ousfield classes.
\newblock {\em Proc. Amer. Math. Soc.}, 124(11):3579--3585, 1996.

\bibitem[HS98]{Hopkins-Smith}
Michael~J. Hopkins and Jeffrey~H. Smith.
\newblock Nilpotence and stable homotopy theory. {II}.
\newblock {\em Ann. of Math. (2)}, 148(1):1--49, 1998.

\bibitem[Joa15]{stroilimi}
Ruth Joachimi.
\newblock Thick ideals in equivariant and motivic stable homotopy categories.
\newblock 2015.
\newblock Available at \url{http://front.math.ucdavis.edu/1503.08456}.

\bibitem[Kuh04]{kuhn}
Nicholas~J. Kuhn.
\newblock Tate cohomology and periodic localization of polynomial functors.
\newblock {\em Invent. Math.}, 157(2):345--370, 2004.

\bibitem[LMB00]{LMB}
G\'erard Laumon and Laurent Moret-Bailly.
\newblock {\em Champs alg\'ebriques}, volume~39 of {\em Ergebnisse der
  Mathematik und ihrer Grenzgebiete. 3. Folge. A Series of Modern Surveys in
  Mathematics [Results in Mathematics and Related Areas. 3rd Series. A Series
  of Modern Surveys in Mathematics]}.
\newblock Springer-Verlag, Berlin, 2000.

\bibitem[LMS86]{LMS86}
L.~Gaunce Lewis, J.~Peter May, and Mark Steinberger.
\newblock {\em Equivariant Stable Homotopy Theory}, volume 1213 of {\em Lecture
  Notes in Mathematics}.
\newblock Springer-Verlag, 1986.

\bibitem[Lur10]{LurieChromaticCourse}
Jacob Lurie.
\newblock Chromatic homotopy theory.
\newblock 2010.
\newblock Available at \url{http://www.math.harvard.edu/~lurie/252x.html}.

\bibitem[Lur11]{DAGVII}
Jacob Lurie.
\newblock Spectral schemes.
\newblock November 2011.
\newblock http://www.math.harvard.edu/~lurie/papers/DAG-VII.pdf.

\bibitem[Lur17]{SAG}
Jacob Lurie.
\newblock Spectral algebraic geometry.
\newblock July 26th 2017.
\newblock http://www.math.harvard.edu/~lurie/papers/SAG-rootfile.pdf.

\bibitem[May96]{May96}
J.~Peter May.
\newblock {\em Equivariant homotopy and cohomology theory}, volume~91 of {\em
  CBMS Regional Conference Series in Mathematics}.
\newblock Published for the Conference Board of the Mathematical Sciences,
  Washington, DC, 1996.
\newblock With contributions by M. Cole, G. Comeza{\~n}a, S. Costenoble, A. D.
  Elmendorf, J. P. C. Greenlees, L. G. Lewis, Jr., R. J. Piacenza, G.
  Triantafillou, and S. Waner.

\bibitem[Mit85]{Mit85}
Stephen~A. Mitchell.
\newblock Finite complexes with {$A(n)$}-free cohomology.
\newblock {\em Topology}, 24(2):227--246, 1985.

\bibitem[MM15]{akhil_affiness}
Akhil Mathew and Lennart Meier.
\newblock Affineness and chromatic homotopy theory.
\newblock {\em J. Topol.}, 8(2):476--528, 2015.

\bibitem[MNN15]{MNN}
Akhil Mathew, Niko Naumann, and Justin Noel.
\newblock Derived induction and restriction theory.
\newblock 2015.
\newblock Available at \url{http://front.math.ucdavis.edu/1507.06867}.

\bibitem[MNN17]{MNN17}
Akhil Mathew, Niko Naumann, and Justin Noel.
\newblock Nilpotence and descent in equivariant stable homotopy theory.
\newblock {\em Adv. Math.}, 305:994--1084, 2017.

\bibitem[Nee92]{Nee92}
Amnon Neeman.
\newblock The chromatic tower for {$D(R)$}.
\newblock {\em Topology}, 31(3):519--532, 1992.
\newblock With an appendix by Marcel B{\"o}kstedt.

\bibitem[Nee01]{neeman}
Amnon Neeman.
\newblock {\em Triangulated categories}, volume 148 of {\em Annals of
  Mathematics Studies}.
\newblock Princeton University Press, Princeton, NJ, 2001.

\bibitem[Rav92]{ravenel-orange}
Douglas~C. Ravenel.
\newblock {\em Nilpotence and periodicity in stable homotopy theory}, volume
  128 of {\em Annals of Mathematics Studies}.
\newblock Princeton University Press, Princeton, NJ, 1992.
\newblock Appendix C by Jeff Smith.

\bibitem[Rez97]{Rez97}
Charles Rezk.
\newblock Notes on the {H}opkins-{M}iller theorem.
\newblock In {\em Homotopy theory via algebraic geometry and group
  representations ({E}vanston, {IL}, 1997)}, volume 220 of {\em Contemp.
  Math.}, pages 313--366. Amer. Math. Soc., 1997.

\bibitem[Sta13]{Sta13}
Nathaniel Stapleton.
\newblock Transchromatic generalized character maps.
\newblock {\em Algebr. Geom. Topol.}, 13(1):171--203, 2013.

\bibitem[Str97]{strickland-formal}
Neil~P. Strickland.
\newblock Finite subgroups of formal groups.
\newblock {\em J. Pure Appl. Algebra}, 121(2):161--208, 1997.

\bibitem[Str12]{stroilova-phd}
Olga Stroilova.
\newblock The {G}eneralized {T}ate {C}onstruction.
\newblock 2012.
\newblock Available at \url{http://web.mit.edu/stroilo/www/main_no_cover.pdf}.

\bibitem[Tat67]{tate}
John~T. Tate.
\newblock {$p$}-divisible groups.
\newblock In {\em Proc. {C}onf. {L}ocal {F}ields ({D}riebergen, 1966)}, pages
  158--183. Springer, Berlin, 1967.

\bibitem[Tho97]{thomason}
R.~W. Thomason.
\newblock The classification of triangulated subcategories.
\newblock {\em Compositio Math.}, 105(1):1--27, 1997.

\bibitem[Ver96]{verdier}
Jean-Louis Verdier.
\newblock Des cat\'egories d\'eriv\'ees des cat\'egories ab\'eliennes.
\newblock {\em Ast\'erisque}, (239):xii+253 pp. (1997), 1996.
\newblock With a preface by Luc Illusie, Edited and with a note by Georges
  Maltsiniotis.

\end{thebibliography}
\end{document}